\documentclass[12pt]{article}%
\usepackage{amsfonts}
\usepackage{amsmath}
\usepackage{amssymb}
\usepackage{graphicx}%
\providecommand{\U}[1]{\protect\rule{.1in}{.1in}}
\newtheorem{theorem}{Theorem}[section]

\newtheorem{corollary}[theorem]{Corollary}

\newtheorem{definition}[theorem]{Definition}
\newtheorem{example}[theorem]{Example}

\newtheorem{lemma}[theorem]{Lemma}

\newtheorem{proposition}[theorem]{Proposition}
\newtheorem{question}[theorem]{Question}
\newtheorem{remark}[theorem]{Remark}

\newenvironment{proof}[1][Proof]{\noindent\textbf{#1.} }{\ \rule{0.5em}{0.5em}}
\begin{document}

\title{Kernels of block Hankel operators and independency of vector-valued functions modulo Nevanlinna class}
\author{Dong-O Kang}
\maketitle

\begin{abstract}
For a  matrix-valued function $\Phi\in L^2_{M_{n\times m}}$, it is well-known that the kernel of a block Hankel operator $H_\Phi$ is an invariant subspace for the shift operator. Thus, if the kernel is nontrivial, then $\ker H_\Phi= \Theta H^2_{\mathbb C^r}$ for a natural number $r$ and an $m\times r$ matrix inner function $\Theta$ by Beurling-Lax-Halmos Theorem.
It will be shown that the size of the matrix inner function $\Theta$ associated with the kernel of a block Hankel operator $H_\Phi$ is closely related with a certain independency of the columns of $\Phi$, which is defined in this paper. As an important application of this result, the shape of shift invariant, or, backward shift invariant subspaces of $H^2_{\mathbb C^n}$ generated by finite elements will be studied.
\end{abstract}


\section{Introduction}

  Kernels of Toeplitz operators on the Hardy space have been studied since 1980's and some of the nice results have been generalized to the block Toeplitz case in 2010's (see \cite{BCD}, \cite{Ch}, \cite{CMP}, \cite{CP1}, \cite{CP2}, \cite{Dy1}, \cite{Dy2}, \cite{Ha}, \cite{Na} and \cite{Sa}). On the other hand, for Hankel operators, M. B. Abrahamse showed in 1976 that the kernel of a Hankel operator on the Hardy space is nontrivial if and only if its symbol function is the quotient of two analytic functions\cite{Ab}. But, there has not been enough research on kernels of block Hankel operators, which is the counterpart of block Toeplitz operators.

It is well known that kernels of block Hankel operators on vector-valued  Hardy space $H^2_{\mathbb C^n}$ are invariant for the shift operator and Beurling-Lax-Halmos Theorem says that an invariant subspace of $H^2_{\mathbb C^n}$ for the shift operator has the form $\Theta H^2_{\mathbb C^r}$ for some natural number $r$ and an $n\times r$ matrix inner function $\Theta$. The main purpose of this paper is to answer the following question:

\bigskip

For a block Hankel operator $H_\Phi$ with a matrix function $\Phi$, how is the size of the matrix inner function $\Theta$ such that $\ker H_\Phi=\Theta H^2_{\mathbb C^r}$ determined?

\bigskip

 We will give a proper answer to the above question and as applications, we will discuss finitely generated shift-invariant subspaces and backward shift-invariant subspaces. The shape of GCD(greatest common inner divisor) and LCM(least common multiple) of matrix inner functions will be discussed as well.
It turns out that the size of the inner function in the above question is determined by a certain independency of the column vectors of the symbol function $\Phi$, which is newly introduced in this paper.

Section 2 is the main part of the paper. \textit{Independency modulo Nevanlinna class} of a set of vector-valued $L^2$-functions on the unit circle will be defined and then it will be shown that the shape of the kernel of a block Hankel operator is closely related to this independency of the column vectors of the symbol matrix function.

In section 3, some results on the preservation of independency modulo Nevanlinna class under matrix multiplication is presented.

In section 4, some concrete examples will be given on kernels of block Hankel operators. These examples show the connection between the block Hankel kernels and independency of the column vectors of its symbol functions modulo Nevanlinna class.

In section 5, as interesting applications of the main result, the shape of finitely generated invariant subspaces of vector-valued Hardy spaces for the shift operator $S$ and its adjoint $S^*$ will be discussed.

\bigskip

Let $\mathbb D$ denote the unit disk in the complex plane and let $\mathbb T=\partial \mathbb D$ be the unit circle.
For $p >0$, $L^p\equiv L^p(\mathbb T)$ denotes the set of all measurable functions $f$ defined on $\mathbb T$ satisfying $$\int_{\mathbb T} |f|^p dm < \infty,$$ where $m$ is the normalized Lebesgue measure on the unit circle. For $p\ge 0$, $H^p\equiv  H^p(\mathbb T)$ denotes the set of all functions $f$ in $L^p$ such that $$\hat{f}(k):= \int_{\mathbb T} z^{-k}f d m =0 \hbox{ for } k < 0.$$ If $p=2$, then $L^2$ and $H^2$ are Hilbert spaces, where the inner product is defined by $$\langle \varphi, \psi \rangle:= \int_{\mathbb T}\varphi \overline \psi dm \hbox{ for } \varphi,\psi \in L^2.$$
For $p=\infty$, $L^\infty\equiv L^\infty(\mathbb T)$ is the set of all essentially bounded measurable functions on $\mathbb T$ and $H^\infty:=L^\infty \cap H^2$.
A function $\varphi$ on the unit circle is said to be of \textit{bounded type}, or, in the Nevanlinna class $\mathcal N$ if there are functions $f,g \in H^\infty$ such that $\varphi =\frac{f}{g}$.
Let's introduce several more notations.

\begin{itemize}
  \item $\mathcal N^p$ : the set of all bounded type functions in $L^p$, that is, $\mathcal N^p = \mathcal N \cap L^p$, where, $0 < p \le \infty$
  \item $L^p_{\mathbb{C}^n}$ and $H^p_{\mathbb{C}^n}$: direct sum of $n$ copies of $L^p$ and $H^p$, respectively
  \item $\mathcal N_{\mathbb C^n}$  and $\mathcal N^p_{\mathbb C^n}$ : direct sum of $n$ copies of $\mathcal N$ and $\mathcal N^p$, respectively
  \item $M_{n\times m}$ : the set of all $n\times m$ matrices with entries in $\mathbb C$
  \item  $L^p_{M_{n\times m}}$,$H^p_{M_{n\times m}}$ and $\mathcal N^p_{M_{n\times m}}$ : the set of all $n\times m$ matrix-valued functions each of whose entries is in $L^p$, $H^p$ and $\mathcal N^p$, respectively
\end{itemize}
   It is well-known that the Nevanlinna class $\mathcal N$ includes $H^p$ for each $p >0$ and each bounded type function $\frac{f}{g}\in L^p$ can be represented as $\frac{f}{g}= \frac{h}{\theta}$, where $\theta$ is an inner function and $h\in H^p$. A function $\varphi \in L^1(\mathbb T)$ defined on the unit circle $\mathbb T$ has a unique harmonic extension onto the unit disk $\mathbb D$. Hence, $\varphi \in L^1(\mathbb T)$  is considered as defined also on the unit disk $\mathbb D$ as its harmonic extension. Therefore, the evaluation of $\varphi \in L^1$ at a point $\alpha \in \mathbb D$ is naturally defined. It is well-known that the radial limit of $\varphi$ coincides with the function value on the unit circle, that is, $\lim_ {\lambda \rightarrow 1^-}\varphi(\lambda\zeta)= \varphi(\zeta)$ for almost all $\zeta \in \mathbb T.$
 In this paper, an element in $L^p_{\mathbb{C}^n}$ will be considered to be a column vector $\left(
                                             \begin{array}{c}
                                               a_1 \\
                                               \vdots \\
                                               a_n \\
                                             \end{array}
                                           \right)
 $ $(a_i \in L^p)$, and it will be denoted by $(a_1, \cdots, a_n)^t$ just to save space, where $(\cdot)^t$ denotes the transpose of a row vector or a matrix.
For a matrix function $\Phi \in L^p_{M_{n\times m}}$, its adjoint $\Phi^*\in L^p_{M_{m\times n}}$ is defined by $$\Phi^*(\zeta)= \Phi(\zeta)^* \hbox{ for } \zeta \in \mathbb T,$$ where, $\Phi(\zeta)^*\in M_{m\times n}$ denotes the adjoint matrix of $\Phi(\zeta) \in M_{n\times m}$. The identity matrix is denoted by $I$ and for $\varphi \in L^p$, $I_\varphi$ denotes the square matrix function whose diagonal entries are $\varphi$ and other entries are $0$. A good property $I_\varphi$ that we frequently make use of in this paper is \begin{equation}\label{commutativity of diagonal}
  I_\varphi A = A I_\varphi
\end{equation} for every $n\times m$ matrix function $A$. Remark in equation (\ref{commutativity of diagonal}) that $I_\varphi$ on the left hand side is an $n\times n$ matrix function and $I_\varphi$ on the right hand side is an $m \times m$ matrix function. $I_\varphi A$ will sometimes be denoted simply by $\varphi A$. A matrix-valued function $\Theta \in H^\infty_{M_{n\times m}}$ is called \textit{inner} if $$\Theta(z)^*\Theta(z)=I \hbox{ for almost all } z\in\mathbb{T},$$ in other words, $\Theta(z)$ is an isometry-valued function almost everywhere on $\mathbb T$. We note that for an $n\times m$ matrix function to be inner, $n \ge m$ should be satisfied. Let $\mathcal{P}_+(\mathbb C ^n)$ denote the set of all analytic polynomials in $H^2_{\mathbb C^n}$. A matrix-valued function $F\in H^2_{M_{n\times m}}$ is called \textit{outer} if $$\hbox{cl}[ F \mathcal{P}_+(\mathbb C^m)] = H^2_{\mathbb C^n}.$$ For an $n \times m$ matrix function to be outer, it is known that $m\ge n$.

To define block Hankel operators we need some preparation. For $\varphi \in L^1$, define the function $P\varphi$ on the unit disk $\mathbb D$ by
 \begin{equation}\label{def of projection}
   P\varphi(z):= \int_{\mathbb T}\frac{\varphi(\zeta)}{1-\overline \zeta z}dm(\zeta),
 \end{equation} for $z\in \mathbb D$, where $m$ is the normalized Lebesgue measure on the unit circle $\mathbb T$. The operator $P$ in (\ref{def of projection}) can be defined also for a vector-valued function $\varphi = \left(
                                                                                \begin{array}{c}
                                                                                  \varphi_1\\
                                                                                  \vdots\\
                                                                                  \varphi_n\\
                                                                                \end{array}
                                                                              \right) \in L^1_{\mathbb C^n}
 $ by the same equation. In this case, it can also be seen as coordinatewise action, that is, $$ P\varphi=\left(
                          \begin{array}{c}
                            P\varphi_1 \\
                            \vdots \\
                            P \varphi_n \\
                          \end{array}
                        \right).
 $$ Note that in (\ref{def of projection}), the function $P\varphi(z)$ was defined for $z\in \mathbb D$, but, it can also be understood as a function on the unit circle $\mathbb T$ using the radial limit process.
If $P$ is restricted to $L^2_{\mathbb C^n}$, then it is the orthogonal projection of $L^2_{\mathbb C^n}$ onto $H^2_{\mathbb C^n}$.

For a matrix-valued function $\Phi \in L^2_{M_{n\times m}}$, the block Hankel operator $H_\Phi$ is defined by
$$
H_\Phi f = J (I-P) (\Phi f)\quad (f\in H^2_{\mathbb{C}^m}),
$$
where the domain of $H_\Phi$ is $\mathcal {D} (H_\Phi)=\{f\in H^2_{\mathbb C^n}| (I-P) (\Phi f) \in \overline {zH^2_{\mathbb C^n}} \}$ and $J$
denotes the operator on $L^2_{\mathbb{C}^n}$ defined by $J(g)(z)= I_{\overline{z}}
g(\overline{z})$ for $g \in L^2_{\mathbb{C}^n}$. It is known that $\mathcal {D} (H_\Phi)$ is a dense subset of $H^2_{\mathbb C^n}$ containing $H^\infty_{\mathbb C^n}$. It is easy to check that if $\Phi \in  L^\infty_{\mathbb C^n}$, then $H_\Phi$ is a bounded operaor and is defined on the whole space $H^2_{\mathbb C^n}$, that is, $\mathcal {D} (H_\Phi)= H^2_{\mathbb C^n}$. If $\Phi$ is a function in $L^2_{M_{n\times m}}$ that is not essentially bounded, then $H_\Phi$ may possibly become an unbounded operator. It is known that the norm of $H_\Phi$ is determined by
$$
||H_\Phi||= \inf\{||\Phi + F||_{\infty}\} : F \in H^2_{M_{n\times m}}\},
$$
where $||\cdot||_{\infty}$ denotes the essential supremum of the matrix norms of the matrix-valued functions\cite{Pe}.

\bigskip

For an analytic matrix function $\Phi\in H^2_{M_{n\times r}}$,
we say that $\Delta\in H^\infty_{M_{n\times m}}$ is a {\it left inner
divisor} of $\Phi$ if $\Delta$
is an inner matrix function such that $\Phi=\Delta A$ for
some $A \in H^{2}_{M_{m\times r}}$ ($m\le n$).

A \textit{shift operator} $S$ on $ H^2_{\mathbb C^n} $ is defined by $$ Sf=I_z f,$$
for $f\in H^2_{\mathbb C^n}$. If we write $f=\left(
                                               \begin{array}{c}
                                                 f_1 \\
                                                 \vdots \\
                                                 f_n \\
                                               \end{array}
                                             \right),
$ where, $f_i\in H^2$, then $Sf = \left(
                                               \begin{array}{c}
                                                 zf_1 \\
                                                 \vdots \\
                                                 zf_n \\
                                               \end{array}
                                             \right)$.
Here, the natural number $n$ is called the multiplicity of the shift operator $S$.
\medskip

On the shape of invariant subspaces for the shift operator $S$, the following theorem is very well-known.
\medskip

\noindent {\bf Beurling-Lax-Halmos Theorem.}\label{beur}
\ \
{\it A
nonzero subspace $M$ of $H^2_{\mathbb C^n}$ is invariant for the
shift operator $S$ on $H^2_{\mathbb C^n}$ if and only if $M=\Theta
H^2_{\mathbb C^m}$, where $m$ is a natural number and $\Theta$ is an inner matrix function in
$H^{\infty}_{M_{n\times m}}\label{H^inftyMnm}$ \hbox{\rm ($m\le
n$)}. \  Furthermore, $\Theta$ is unique up to a unitary constant
right factor; that is, if $M=\Delta H^2_{\mathbb{C}^r}$ {\rm (}for
$\Delta$ an inner function in $H^{\infty}_{M_{n\times r}}${\rm )}, then
$m=r$ and $\Theta=\Delta W$, where $W$ is a unitary matrix mapping
$\mathbb C^m$ onto $\mathbb C^m$. \  }

\bigskip

 It is well known that for $\Phi\in L^2_{M_{n\times m}}$, the Hankel operator $H_{\Phi}$ satisfies
\begin{equation}\label{hankel equation}
H_\Phi S=S^*H_\Phi.
\end{equation}
Note in this equation that the multiplicity of $S$ is $m$ on the left hand side and $n$ on the right hand side.
From equation (\ref{hankel equation}), we find that the kernel
of a Hankel operator $H_\Phi$ is an invariant subspace for the
shift operator $S$ on $H^2_{\mathbb C^m}$.   Thus, if
$\hbox{ker}\,{H_{\Phi}} \ne (0)$, then by Beurling-Lax-Halmos
Theorem,
$$
\hbox{ker}\, H_\Phi=\Theta H^2_{\mathbb{C}^r}
$$
for some natural number $r \,(\,\le m\,)$ and an inner matrix function $\Theta \in H^2_{M_{m\times r}}$.  Our concern is how this natural number $r$ is determined by the matrix symbol function $\Phi$.

\medskip

\section{Independency modulo Nevanlinna class and shape of kernels of block Hankel operators}

In 1970s, M. B. Abrahamse showed \cite{Ab} that for $\varphi\in L^\infty$,
\begin{equation}\label{kernel of Hankel-scalar}
  \ker H_\varphi \ne (0) \Longleftrightarrow \varphi \hbox{ is of bounded type}.
\end{equation}
 In 2006, (\ref{kernel of Hankel-scalar}) was partially generalized to the case of matrix-valued symbol case in \cite{GHR}.
\begin{proposition}(\cite{GHR})
   Let $\Phi \in L^\infty_{M_{n\times n}}$, then $\ker H_\Phi =\Theta H^2_{\mathbb C^n}$ for a square matrix inner function $\Theta\in H^\infty_{M_{n\times n}}$ if and only if each entry function of $\Phi$ is of bounded type.
\end{proposition}

\medskip
This proposition, together with Beurling-Lax-Halmos Theorem, implies that if $\Phi\in L^\infty_{M_{n\times n}}$ has at least one entry function that is not of bounded type and if $H_\Phi$ has a nontrivial kernel, then $\ker H_\Phi= \Theta H^2_{\mathbb C^m}$ for a nonsquare matrix inner function $\Theta \in  H^\infty_{M_{n\times m}}$.

As was mentioned in the introduction, the main result of this section, which also is the main result of the paper, shows that the size of the inner function $\Theta$ associated with the kernel of a block Hankel operator is determined by a certain independency of the columns of the symbol function. For that purpose, ``independence modulo Nevanlinna class'' will be defined and some of its basic properties will be derived in this section.
For $\Phi \in L^2_{M_{n\times m}}$, define $\hbox{Rank} \Phi$ to be the essential supremum of the rank of $\Phi(\zeta)$, where, $\zeta \in \mathbb T$, that is,
$$
\hbox{Rank}\,\Phi:=\hbox{ess}\sup_{\zeta \in \mathbb
T}\,\hbox{rank}\Phi(\zeta),
$$
where, $\hbox{rank}\Phi(\zeta)$ denotes the rank of $\Phi(\zeta)$ as a matrix, that is, $\hbox{rank}\Phi(\zeta):=\hbox{dim} (\Phi(\zeta)\mathbb
C^m).$
\begin{lemma}\label{size of inner function}
  For a nontrivial shift-invariant subspace $M$ of $H^2_{\mathbb C^n}$, let $m$ be defined by
   $$m= \max\{\hbox{Rank} [f_1, \cdots , f_r] : f_i \in M \hbox{ and } r \in \mathbb N \}.$$ Then $M= \Theta H^2_{\mathbb C^m}$ for some $n \times m$ matrix-valued inner function $\Theta \in H^\infty_{M_{n\times m}}$.
\end{lemma}

\begin{proof}
By Beurling-Lax-Halmos Theorem, there exist a natural number $l(\leq n)$ and a matrix-valued inner function $\Theta \in H^\infty_{M_{n\times l}}$ such that $M= \Theta H^2_l.$ We need to prove $l=m$. Let $F:=[f_1, \cdots , f_r]$. Since $\hbox{Rank}F$ cannot exceed $n$, it suffices to consider the case $r=n$.
Let $F\in H^2_{M_{n\times n}}$ be a matrix-valued function all of whose columns lie in $M.$ Clearly, $F=\Theta G$ with $G\in H^2_{M_{l\times n}}$. Let $g_j$ be the $j$-th column of $G$. Consider the case where $g_j= e_j$ for $j=1,\cdots,l$ and $g_j=0$ for $ l < j \le n$, then we may write $F=[\Theta, 0]$, hence, $\hbox{Rank}F=l$. Thus, we have $m \ge l$. On the other hand, $m \le l$ is obvious that because for each $\zeta \in \mathbb T$, $\hbox{rank}F(\zeta)=\hbox{rank}(\Theta(\zeta)G(\zeta))\le \hbox{rank} \Theta(\zeta) =l$.
\end{proof}

\bigskip

The following definition will play the key role throughout the paper.

\begin{definition}
  A finite set of vector-valued functions $\{\varphi_1,\varphi_2,\cdots, \varphi_r \} \subset L^2_{\mathbb C ^n}$ is defined to be ``independent modulo Nevanlinna class'' if for bounded type functions $a_1,\cdots, a_r$ in $\mathcal{N}$, $a_1\varphi_1 + \cdots+ a_r\varphi_r\in \mathcal{N}_{\mathbb C^n}$ implies $a_1=\cdots=a_r=0$.
\end{definition}
We will also say that the vector-valued functions $\varphi_1,\cdots, \varphi_r $ are independent modulo Nevanlinna class when the set $\{\varphi_1,\cdots, \varphi_r \} $ is.

\bigskip
The next proposition will be used frequently in the paper, often, without mentioning.
\begin{proposition}\label{equivalent cond of ind mod N}
$\varphi_1, \cdots, \varphi_r\in L^2_{\mathbb C^n}$ are independent modulo Nevanlinna class if and only if for analytic functions $f_1, f_2, \cdots, f_r \in H^2$, $f_1\varphi_1+f_2\varphi_2+\cdots+f_r\varphi_r \in H^1_{\mathbb C^n}$ implies $f_1=\cdots=f_r=0$, in other words, $\ker H_\Phi= (0)$ where $\Phi=[\varphi_1,\cdots,\varphi_r]\in L^2_{M_{n \times r}}$.
\end{proposition}

\begin{proof}
Necessity($\Longrightarrow$) is immediate from the definition above because $H^2\subset \mathcal N$ and $H^1_{\mathbb C^n}\subset \mathcal N_{\mathbb C^n}$.

To show sufficiency($\Longleftarrow$), assume that $f_1\varphi_1+f_2\varphi_2+\cdots+f_r\varphi_r \in H^1_{\mathbb C^n}$ implies $f_1=\cdots=f_r=0$ for analytic functions $f_1, f_2, \cdots, f_r \in H^2$.
Now suppose
\begin{equation}\label{equi condition for indp 2}
  a_1\varphi_1 + \cdots+ a_r\varphi_r= b,
\end{equation}
where $a_i, b\in \mathcal N$.
Since $a_i, b$ are of bounded type, we may write $a_i=\frac{h_i}{g_i}$ and $b=\frac{h}{g}$, where, $g_i,g,h_i,h \in H^\infty$ and $g_i,g\ne 0$.
Multiplying equation (\ref{equi condition for indp 2}) by $g\prod_{j=1}^r g_j$, we have

\begin{equation}\label{equi condition for indp 3}
  \frac {h_1 g\prod_{j=1}^r g_j}{g_1} \varphi_1 + \cdots +  \frac {h_r g \prod_{j=1}^r g_j}{g_r}\varphi_r= h\prod_{j=1}^r g_j.
\end{equation}
Note that $\frac {h_i g\prod_{j=1}^r g_j}{g_i}\in H^\infty\subset H^2$ for each $i=1,\cdots, r$. By assumption, we have $\frac {h_i g\prod_{j=1}^r g_j}{g_i}=0$ for each $i$, which implies $h_i=0$ because $g_i$ and $g$ are nonzero analytic functions. Therefore, $a_i=0$ for each $i$, which implies $\varphi_1,\cdots, \varphi_r$ are independent modulo Nevanlinna class.
\end{proof}

\bigskip

Throughout this paper, independence of a set of vector-valued functions denotes independence modulo Nevanlinna class if not specified otherwise and the phrase ``independent modulo Nevanlinn class'' will often be written as ``independent mod $\mathcal N$'' for simplicity.
For an arbitrary set $A$ of functions in $L^2_{\mathbb C^n}$, we define the \textit{degree of independence}, or, \textit{independency} of $A$ modulo Nevanlinna class, which will be denoted by ``$\hbox{ind}_{\mathcal N} A$'', by
$$\hbox{ind}_{\mathcal N} A := \sup \{|B|: B \hbox{ is a finite subset of } A \hbox{ that is independent mod } \mathcal N\},$$
where $|B|$ denotes the order of the set $B$.  Note that it is possible to have $\hbox{ind}_{\mathcal N} A =\infty$ if $A$ is an infinite set. If $A$ has no independent subset mod $\mathcal N$, then we naturally define $\hbox{ind}_{\mathcal N} A =0$.

\bigskip

Using Proposition \ref{equivalent cond of ind mod N}, it is easily verified that if $b_1,\cdots, b_r \in \mathcal N^\infty$ are nonzero functions and $\varphi_1, \cdots, \varphi_r\in L^2_{\mathbb C^n}$, then
\begin{equation}\label{independency multiplied by diagonal matrix}\varphi_1, \cdots, \varphi_r \hbox{ are independent mod }\mathcal N \Longleftrightarrow b_1\varphi_1, \cdots, b_r \varphi_r \hbox{ are independent mod } \mathcal N.
\end{equation}

 In fact, more is true. In the proof of the following proposition, $A_{adj}$ denotes the classical adjoint of the matrix $A$. As was mentioned before, Proposition \ref{equivalent cond of ind mod N} will be used often in the sequel, without mentioning.

\begin{proposition}\label{preservation of independence}
  Let $A$ be an $r\times r$ matrix-valued function with entries in $\mathcal N^\infty$ such that $\det A \ne 0$ and $\Phi= [\varphi_1 \quad \cdots \quad \varphi_r] \in L^2_{M_{m \times r}}$. Let $\psi_i$ denote the $i$-th column of $\Psi:= \Phi A$, then $\varphi_1, \cdots, \varphi_r$ are independent modulo Nevanlinna class if and only if $\psi_1, \cdots, \psi_r$ are independent modulo Nevanlinna class.
\end{proposition}

\begin{proof}
 To prove necessity($\Longrightarrow$), assume that the columns of $\Phi$ are independent modulo Nevanlinna class. Suppose now, $a_1\psi_1+\cdots+a_r \psi_r\in \mathcal N^1_{\mathbb C^m}$, or, equivalently, $\Phi A a \in \mathcal N^1_{\mathbb C^m}$, where $a:=(a_1 , \cdots , a_r)^t$. Write $Aa:=b =(b_1, \cdots, b_r)^t
  \in L^2_{\mathbb C^r}$, then it is clear that $b_i \in \mathcal N^2$ for each $i$ and
  $$
  b_1\varphi_1+\cdots+b_r \varphi_r\in \mathcal N^1_{\mathbb C^m}\, (\hbox { from }\Phi A a \in \mathcal N^1_{\mathbb C^m}).
  $$
 Since $\varphi_i$ are independent mod $\mathcal N$, we have $b_i=0$ for each $i=1,\cdots,r$.
   On the other hand, since each entry function of $A$ is in $\mathcal N^\infty$, it is clear that $\det A = \frac{h}{g}$ for some $g,h\in H^\infty$. Thus, $\det A \ne 0$ implies $\det A(\zeta)\ne0$ almost everywhere on $\mathbb T$. Therefore, $Aa=b=0$ implies $a_i=0$ for each $i$, which implies $\psi_i$ are independent modulo Nevanlinna class.

 To prove sufficiency($\Longleftarrow$), assume $\psi_1,\cdots, \psi_r$ are independent mod $\mathcal N$. Now suppose $a_1 \varphi_1 + \cdots + a_r \varphi_r \in \mathcal N^1_{\mathbb C^m}$ for $a_i \in \mathcal N^2$, or, equivalently, $\Phi a \in \mathcal N^1_{\mathbb C^m}$ where, $a=(a_1,\cdots, a_r)^t$. Let $ a^\prime=(a_1^\prime,\cdots, a_r^\prime)^t:= A_{adj}a$, then it is clear that each $a_i^\prime \in \mathcal N^2$. Since $(\det A) \in N^\infty$, $\Phi a \in \mathcal N^1_{\mathbb C^m}$ implies $$ \Psi a^\prime=\Phi A A_{adj}a = \Phi I_{(\det A) }a = (\det A)\Phi a \in \mathcal N^1_{\mathbb C^m}.$$
 Since each $a_i^\prime \in \mathcal N^2$, independence of $\{\psi_1,\cdots, \psi_r\}$ (mod $\mathcal N$) implies $a^\prime_1=\cdots= a^\prime_r= 0$, i.e., $A_{adj}a=0$. Observing that $\det A_{adj}=(\det A)^{r-1}\ne0$ almost everywhere on $\mathbb T$, we find $a=0$, that is, $a_1=\cdots=a_r=0$. Thus, we conclude that $\varphi_1,\cdots,\varphi_r$ are independent mod $\mathcal N$.
\end{proof}

\bigskip

\begin{proposition}\label{preservation of independence under left multiplication}
  Let $A$ be an $m\times m$ matrix-valued function with entries in $\mathcal N^\infty$ such that $\det A \ne 0$ and $\Phi= [\varphi_1, \cdots, \varphi_r] \in L^2_{M_{m \times r}}$. Let $\psi_i$ denote the $i$-th column of $\Psi:= A \Phi$, then $\varphi_1, \cdots, \varphi_r$ are independent modulo Nevanlinna class if and only if $\psi_1, \cdots, \psi_r$ are independent modulo Nevanlinna class.
\end{proposition}

\begin{proof}
To prove necessity($\Longrightarrow$), assume $\varphi_1,\cdots, \varphi_r $ is independent mod $\mathcal N.$
   To show that $\psi_1,\cdots, \psi_r$ are independent mod $\mathcal N$, let $a_i \in \mathcal N^2$ and $a_1\psi_1+\cdots + a_r \psi_r \in \mathcal N^1_{\mathbb C^r}$. Set $a:=(a_1,\cdots, a_r)^t\in \mathcal N^2_{\mathbb C^r}$, then $\Psi a = a^\prime$ for some $a^\prime \in \mathcal N^1_{\mathbb C^m}$. Let $b=I_{(\det A)}a$, then, since $A_{adj}A=A A_{adj}= I_{(\det A)}$,

  \begin{equation}\label{multiplication by classical adjoint}
  \Phi b = \Phi I_{(\det A)}a= I_{(\det A)} \Phi a= A_{adj}\Psi a= A_{adj} a^\prime.
  \end{equation}

  Note that $b= I_{(\det A)}a \in \mathcal N^2_{\mathbb C^m}$ and $A_{adj}a^\prime \in \mathcal N^1_{\mathbb C^m}$ since $\det A, A_{adj} \in \mathcal N^\infty$.
  Since we assumed that the columns of $\Phi$ are independent mod $\mathcal N$, equation (\ref{multiplication by classical adjoint}) gives $b=0$, which implies $a=(a_1,\cdots, a_r)^t=0$ because $\det A\ne 0$ almost everywhere on $\mathbb T$. Therefore, we conclude that $\psi_1,\cdots,\psi_r$ are independent mod $\mathcal N$.

To prove sufficiency ($\Longleftarrow$), assume $\psi_1,\cdots,\psi_r$ are independent mod $\mathcal N$. Suppose $a_1\varphi_1+\cdots + a_r \varphi_r \in \mathcal N^1_{\mathbb C^m}$, that is, $\Phi a= a^\prime$, where  $a=(a_1, \cdots, a_r)^t \in \mathcal N^2_{\mathbb C^r}$ and $a^\prime \in \mathcal N^1_{\mathbb C^m}$. Hence, we have $$ a_1\psi_1 + \cdots, + a_r\psi_r= \Psi a = A \Phi a= A a^\prime \in N^1_{\mathbb C^m}$$ because $A \in \mathcal N^\infty_{M_{m}}$ and $a^\prime \in \mathcal N^1_{\mathbb C^m}$. Since we assumed that the columns of $\Psi$ are independent mod $\mathcal N$, we have $a=(a_1,\cdots, a_r)^t=0$. Therefore, we conclude that the columns $\varphi_1,\cdots, \varphi_r$ of $\Phi$ are independent mod $\mathcal N.$
\end{proof}

\bigskip

Now we are ready for the main theorem of the paper, which states that the shape of the kernels of Hankel operators with matrix-valued functions is closely related to the maximum number of independent columns(mod $\mathcal N$) of the symbol functions.

\begin{theorem}\label{ker of block Hankel}
Let $\Phi=[\varphi_1 ,\cdots ,\varphi_m]$ be an $n\times m$ matrix-valued function ( $\varphi_j \in L^2_{\mathbb C^n}$ ) and let $r:=\hbox{ind}_\mathcal N\{\varphi_1,\cdots, \varphi_m\}$. If $r<m$, then $\ker H_{\Phi}=\Theta H^2_{\mathbb C^{m-r}}$ for some $m \times (m-r)$ matrix inner function $\Theta \in H^\infty_{M_{m\times {m-r}}}$ and if $m=r$, then we have $\ker H_\Phi=(0)$.
\end{theorem}

\begin{proof}
If $r=m$, then since the set of all columns of $\Phi$ is independent mod $\mathcal N$, by Proposition \ref{equivalent cond of ind mod N} we have $\ker H_\Phi=0$.
Now assume that $r<m$.
Without loss of generality, we may assume that the first $r$ columns $\varphi_1,\varphi_2, \cdots, \varphi_r$ of $\Phi$ are independent modulo Nevanlinna class. Since $\hbox{ind}_{\mathcal N}\{\varphi_1,\cdots, \varphi_m\}=r$, for any index $j>r$, the $(r+1)$ columns $\varphi_1, \varphi_2, \cdots, \varphi_r,\varphi_j$ are not independent mod $\mathcal N$.

Set $l:=m-r$ and let $f_1=(f_{11}, \cdots , f_{1m})^t, \cdots, f_{l+1}=(f_{l+1,1},\cdots , f_{l+1,n})^t$ be $(l+1)$ vectors in $\ker H_\Phi$. We will show that the vectors $f_1,\cdots,f_{l+1}$ are not independent mod $\mathcal N.$\\
Let $$F:=[f_1, f_2,\cdots, f_{l+1}]\in H^2_{M_{m\times (l+1)}}$$ and $$G:=\left[
                                                   \begin{array}{ccc}
                                                     f_{1,r+1} & \cdots & f_{l+1,r+1} \\
                                                     \vdots &  & \vdots \\
                                                     f_{1m} & \cdots & f_{l+1,m} \\
                                                   \end{array}
                                                 \right]
.$$ Obviously, $G$ is an $l\times (l+1)$ matrix function. Let $G_j$ denote the $l \times l$ matrix function obtained from $G$ by eliminating its $j$-th column. Consider a vector-valued function \begin{equation}\label{f}
  f:=\sum_{j=1}^{l+1}(-1)^{j+1}(\hbox{det}G_j) f_j \in H^2_{\mathbb C^m},
\end{equation} then the $i$-th coordinate of the column vector $f$ is \begin{equation}\label{g_i}
g_i:=\sum_{j=1}^{l+1}(-1)^{j+1}(\hbox{det}G_j) f_{ji}.
\end{equation} Define $F_i$ to be the $(l+1)\times(l+1)$ matrix function obtained by placing the $i$-th row of $F$ on the top of the $l\times(l+1)$ matrix function $G$, then the summation in (\ref{g_i}) reduces to det$F_i$, that is, $g_i= \hbox{det} F_i$. Note that if $i>r$, then $\hbox{det} F_i=0$, the zero function, because for those $i$'s, $F_i$ has two identical rows. Therefore we have $$f=(  g_1,\cdots, g_r, 0, \cdots, 0)^t
.$$ Recall that each $f_j$ is in $\ker H_\Phi$ (i.e., $\Phi f_j \in H^1_{\mathbb C^n}$) and note that $\det G_j\in H^p$ for some $p>0$. Thus, from (\ref{f}), we know that $\Phi f \in H^{p^\prime}_{\mathbb C^n}$ for some $p^\prime>0$. Hence we have $$g_1\varphi_1+\cdots + g_r \varphi_r \in H^{p^\prime}_{\mathbb C^n}$$ for some $p^\prime >0$. Since we assumed that $\varphi_1, \cdots, \varphi_r$ are independent modulo Nevanlinna class, we have $g_1=\cdots= g_r=0$ since $H^p, H^{p^\prime} \subset \mathcal N$. Hence, $$\hbox{det}F_i=0 \hbox{ for all } i= 1, \cdots, m.$$ Thus, we conclude that $\hbox{rank}F(\zeta)\le l$ for almost all $\zeta \in \mathbb T$, hence $\hbox{Rank}F \le l$. Recall that the columns $f_j$ of $F$ were arbitrarily chosen $(l+1)$ vectors from $\ker H_\Phi$. By Lemma \ref{size of inner function}, the number $s$ of the columns of the inner function $\Theta$ is less than or equal to $l(=m-r)$, that is, $s \leq l$.

To show the opposite inequality, $s\geq l$, recall that for each $i > r$, the $(r+1)$ vectors $\varphi_1, \cdots, \varphi_r, \varphi_i$ are not independent modulo Nevanlinna class while $\varphi_1, \cdots, \varphi_r$ are. By Proposition \ref{equivalent cond of ind mod N}, there exist analytic functions $h_{i1}, \cdots, h_{ir}, h_{ii}\in H^2$ such that $$h_{i1}\varphi_1 + \cdots + h_{ir} \varphi_r +h_{ii} \varphi_i \in H^1_{\mathbb C^n} \hbox{ and } h_{ii} \neq 0.$$ If we define $h_i$, for each $r<i \le m$, by $$h_i:=(h_{i1},\cdots, h_{ir},0,\cdots, 0)^t + h_{ii}e_i \in H^2_{\mathbb C^m},$$
then, $h_i \in \ker H_\Phi$, where, $e_i$ denotes the $i$-th unit vector in $\mathbb C^m$. Consider the $m\times l$ matrix-valued analytic function $F:=[h_{r+1}, h_{r+2},\cdots, h_{m}]$. Recalling that the $i$-th entry of $h_i\,(i \ge r+1)$ is a nonzero analytic function, it is clear that $\hbox{rank}F(\zeta)=l$ almost everywhere on $\mathbb T$. By Lemma \ref{size of inner function}, we have $s\geq l$ as we wanted.
The proof is complete.
\end{proof}

\bigskip

One can notice that the above proof is still valid even if the condition on the set of columns $\{\varphi_1,\cdots, \varphi_r\}$ is weakened to be a \textit{maximal} independent subset (mod $\mathcal N$) of $\{\varphi, \cdots, \varphi _m\}$. So we have the following corollary.

\begin{corollary}\label{number of maximal bddly indep col}
For a finite set of vector-valued functions $A:=\{\varphi_1,\cdots,\varphi_m \}\subset L^2_{\mathbb C^n}$, every maximal independent subset (mod $\mathcal N$) of $A$ has the same order.
\end{corollary}

\begin{proof}
Let $\Phi:=[\varphi_1, \cdots , \varphi_m ]$ be an $n\times m$ matrix-valued function. By Beurling-Lax-Halmos theorem, $\ker H_\Phi = (0)$, or, $\ker H_\Phi= \Theta H^2_{\mathbb C^s}$ for some natural number $s$ and an inner function $\Theta \in H^\infty_{M_{n\times s}}$. Since the proof of Theorem \ref{size of inner function} is still valid even if $\{\varphi_1,\cdots, \varphi_r \}$ is a maximal set of independent columns of $\Phi$ mod $\mathcal N$, we can conclude that every maximal independent set of columns of $\Phi$ has order $r=m-s$. The proof is complete.
\end{proof}

\bigskip

According to the above corollary, for a finte set $A$ of vector-valued functions in $L^\infty_{\mathbb C^n}$, the \textit{independency} of $A$ modulo Nevanlinna class can be understood as the order of an arbitrary maximal independent subset of $A$ (mod $\mathcal N$).
Here's another analogue of independence modulo Nevanlinna class to `linear independence'.

\begin{corollary}\label{fixedness of order of max indep set}
Let $A= \{\varphi_1, \varphi_2,\cdots, \varphi_r\}$ and $B=\{\psi_1, \psi_2, \cdots, \psi_s\}$ be two sets of independent vectors mod $\mathcal N$ in $L^2_{\mathbb C^n}$ with $r<s$. Then there exists a subset $B_1=\{\psi_{i_1}, \cdots, \psi_{i_{s-r}}\}$ of $B$ such that $A \cup B_1= \{\varphi_1,\cdots, \varphi_r, \psi_{i_{1}}, \cdots, \psi_{i_{s-r}}\}$ is independent modulo Nevanlinna class.
\end{corollary}

\begin{proof}
Consider the set
$$
C:=A\cup B =\{\varphi_1,\cdots, \varphi_r,\psi_1,\cdots, \psi_s\}
$$ with $r+s$ elements and let $\alpha:= \hbox{ind}_{\mathcal N} C $.
Since $B$ is an independent subset(mod $\mathcal N$) of $C$, of order $s$, it is clear that $\alpha \ge s > r$.
Note that $A$ is an independent subset of $C$. Since Corollary \ref{number of maximal bddly indep col} says that any maximal independent subset of $C$ has order $\alpha$,
there should be a maximal independent subset of $C$ of order $\alpha$, containing $A$. Hence, the conclusion follows.
\end{proof}

\bigskip

\section{Preservation of independency of matrix columns under multiplication by matrix functions}

In this section, it will be shown that the independency of columns of a matrix function is preserved under multiplication by matrix functions with certain properties.
The next result, which is a generalization of Proposition \ref{preservation of independence}, shows that the independency(mod $\mathcal N$) of columns of a matrix function is preserved under right multiplication by matrix functions with bounded type entries and nonzero determinants. Preservation of independency(mod $\mathcal N$) under left multiplication will be discussed later in this section.

\begin{theorem}\label{independency multiplied by matrix_0}
Let $\Phi=[\varphi_1,\cdots , \varphi_s]\in L^2_{M_{n\times s}}$, where $\varphi_i \in L^2_{\mathbb C^n}$ and let $A=[a_{ij}]_{i,j=1}^s$ be an $s \times s$ matrix function whose entries are in $\mathcal N^\infty$ and $\det A \ne 0$. Set $\Psi=[\psi_1,\cdots ,\psi_s]:=\Phi A$, where, $\psi_i \in L^2_{\mathbb C^n}$, then
$$\hbox{ind}_\mathcal{N}\{\varphi_1,\cdots, \varphi_s\}=\hbox{ind}_{\mathcal N}\{\psi_1, \cdots, \psi_s\}.$$
\end{theorem}

\begin{proof}
Let $r:=\hbox{ind}_\mathcal{N}\{\varphi_1,\cdots, \varphi_s\}$ and $r^\prime :=\hbox{ind}_{\mathcal N}\{\psi_1, \cdots, \psi_s\}$.
Referring to Theorem \ref{ker of block Hankel}, we may let $\ker H_{\Phi}= \Theta H^2_{\mathbb C^{s-r}}$ and $\ker H_{\Psi}= \Theta^\prime H^2_{\mathbb C^{s-r^\prime}}$, where, $\Theta \in H^\infty_{M_{s\times(s-r)}}$ and $\Theta^\prime \in H^\infty_{M_{s\times(s-r^\prime)}}$ are matrix inner functions. Since each entry of $A_{adj}$ belongs to $\mathcal N^\infty$ by elementary matrix theory, there exist inner functions $\theta_1, \theta_2\in H^\infty$ such that both $A I_{\theta_1}$ and $A_{adj} I_{\theta_2} $ belong to $ H^\infty_{M_{s\times s}}$.
Define $\theta=\theta_1 \theta_2$, then
$$
I_{(\det A)\theta}=A A_{adj}I_\theta= (A I_{\theta_1})(A_{adj}I_{\theta_2})\in H^\infty_{M_{s\times s}}.
$$
Observe, for $f \in H^2_{\mathbb C^{s-r}} $,
\begin{equation}\label{right multiplication 1}
\Phi A A_{adj} I_{\theta} \Theta f = \Phi I_{(\det A)\cdot\theta}\Theta f =\Phi \Theta  I_{(\det A)\theta} f \in H^1_{\mathbb C^{n}}
\end{equation} because $\ker H_\Phi =\Theta H^2_{\mathbb C^{s-r}}$ and $I_{(\det A)\theta} f \in H^2_{\mathbb C^{s-r}}$.
  Since $A_{adj} I_\theta \Theta f \in H^2_{\mathbb C^{s}}$, (\ref{right multiplication 1}) implies
   \begin{equation}\label{preservation1}
      A_{adj} I_\theta \Theta H^2_{\mathbb C^{s-r}} \subset \ker H_{\Phi A}= \ker H_\Psi.
   \end{equation}

In particular, $A_{adj} I_\theta \Theta {e}_i \in \ker H_{\Psi}$, where, $e_i$ is the $i$-th unit vector in $\mathbb C^{s-r}$. Observe that
$$
(\det A_{adj}(\zeta))= (\det A(\zeta))^{s-1}\ne 0
$$
for almost all $\zeta \in \mathbb T$ since $\det A$ is a nonzero function of bounded type. Consider a matrix function
   $$
   F:=A_{adj} I_\theta \Theta [{e}_1 \, \cdots \, {e}_{s-r}],
   $$ all of whose columns are in $\ker H_{\Psi}$, then we have $\hbox{rank} F(\zeta)= s-r$ almost everywhere on $\mathbb T$. Hence $\hbox{Rank}F= s-r$. By Lemma \ref{size of inner function}, we conclude that
   \begin{equation}\label{right multiplication 2}
   s - r^\prime\ge s-r, \hbox{ or, equivalently, } r^\prime \le r.
   \end{equation}

For the opposite inequality, $r^\prime \ge r$, recall that $\ker H_\Psi=\ker H_{\Phi A}= \Theta^\prime H^2_{\mathbb C^{s-r^\prime}}$ and observe
\begin{equation}\label{right multiplicatio 2-1}
\Phi I_{\theta_1} A \Theta^\prime g=\Phi A \Theta^\prime I_{\theta_1}g \in H^1_{\mathbb C^n}.
\end{equation}
Since $(I_{\theta_1}A) \Theta^\prime g \in H^2_{\mathbb C^s}$, (\ref{right multiplicatio 2-1}) implies $I_{\theta_1} A \Theta^\prime g \in \ker H_\Phi$ for each $g\in H^2_{\mathbb C^{s-r^\prime}}$. Consider a matrix function
$$
F^\prime:=I_{\theta_1} A \Theta^\prime [e_1,\cdots, e_{s-r^\prime}],$$
 where $e_i$ is the $i$-th unit vector in $\mathbb C^{s-r^\prime}$. By the same argument as the case of $\hbox{Rank}F$, we have $\hbox{Rank} F^\prime =s-r^\prime.$ Since the columns of $F^\prime$ are all in $\ker H_{\Phi}$, using Lemma \ref{size of inner function} again, we conclude
\begin{equation}\label{right multiplication 3}
 s-r \ge s-(s-r^\prime), \hbox{ or, equivalently, } r^\prime \ge r.
\end{equation}
Inequalities (\ref{right multiplication 2}) and (\ref{right multiplication 3}) gives $r=r^\prime$, as desired.
\end{proof}
\medskip

\begin{corollary}\label{independency multiplied by matrix_1}
Let $\Phi=[\varphi_1,\cdots , \varphi_s]\in L^2_{M_{n\times s}}$, where $\varphi_i \in L^2_{\mathbb C^n}$ and let $A=[a_{ij}] \in L^\infty_{M_{s\times l}}$, where, $a_{ij}\in \mathcal N^\infty$, $s \le l$ and $\hbox{Rank} A =s$. Set $\Psi=[\psi_1,\cdots,\psi_l]:=\Phi A$, then the columns of $\Phi$ and $\Psi$ have the same independency, that is,
$$\hbox{ind}_\mathcal{N}\{\varphi_1,\cdots, \varphi_s\}=\hbox{ind}_{\mathcal N}\{\psi_1, \cdots, \psi_l\}.$$
\end{corollary}

\begin{proof}
Let $A =[a_1,\cdots, a_l]$, where $a_i$ is the $i$-th column of $A$.
Since $\hbox{Rank} A =s$, there are $s$ columns $a_{i_1}, \cdots, a_{i_s}$ of $A$ such that $\det[a_{i_1}\, \cdots \, a_{i_s}](\zeta)\ne 0$ on a subset of $\mathbb T$ of positive measure. Since $\det[a_{i_1}, \cdots , a_{i_s}]$ is of bounded type, we find $\det[a_{i_1}, \cdots , a_{i_s}](\zeta)\ne 0$ almost everywhere on $\mathbb T$. Without loss of generality, we may assume that $ \det [a_1 ,\cdots, a_s]\ne 0$ a.e. on $\mathbb T$. Let $A^\prime:=[a_1 ,\cdots, a_s]$, $\Psi^\prime:=\Phi A^\prime$ and $b_p=\left(
                     \begin{array}{c}
                       b_{p1} \\
                       \vdots \\
                       b_{ps} \\
                     \end{array}
                   \right)
:= A^\prime_{adj}a_p \in \mathcal N^\infty_{\mathbb C^s}$ for each $1 \le p \le l$.  Observe the following :
\begin{equation}\label{independency cor 1}
\sum_{i=1}^{s}b_{pi}\psi_i= \Psi^\prime b=\Phi A^\prime A^\prime_{adj} a_p=\Phi I_{(\det A^\prime)}a_p=(\det A^\prime)\Phi a_p=(\det A^\prime) \psi_p.
\end{equation}

Now let $B=\{\psi_{j_1}, \cdots, \psi_{j_r}\}$ be a maximal independent subset(mod $\mathcal N$) of $\{\psi_1,\cdots, \psi_s\}$, then by the definition of independence modulo Nevanlinna class, it is clear that for each $1\le k \le s$
\begin{equation}\label{independency cor 2}
  \psi_k=\sum_{\psi_i\in B}h_i \psi_i,\hbox{ where, $h_i\in \mathcal N$}.
\end{equation}
Note that $b_{pi}\in \mathcal N^\infty$ and $\det A^\prime(\ne 0)\in \mathcal N^\infty$ in (\ref{independency cor 1}), hence, by (\ref{independency cor 1}) and (\ref{independency cor 2}), for each $1 \le k \le l$,
$$\psi_k=\sum_{\psi_i\in B}h_i^\prime \psi_i, \hbox{ where, $h_i^\prime \in \mathcal N$.} $$

Therefore, we conclude that $B$ is a maximal independent subset (mod $\mathcal N$) also of $\{\psi_1,\cdots,\psi_l\}$. Hence,
\begin{equation}\label{preservation of independence 2}
 \hbox{ind}_{\mathcal N}\{\psi_1, \cdots, \psi_l\}
 =|B|
 =\hbox{ind}_{\mathcal N}\{\psi_1, \cdots, \psi_s\},
\end{equation}
where, $|B|$ denotes the order of the set $B$.
 But, since $[\psi_1,\cdots ,\psi_s]=[\varphi_1,\cdots,\varphi_s] A^\prime$ and $\det A^\prime \ne 0$,
 \begin{equation}\label{preservation of independence 3}
   \hbox{ind}_{\mathcal N}\{\psi_1, \cdots, \psi_s\}= \hbox{ind}_{\mathcal N}\{\varphi_1, \cdots, \varphi_s\}
 \end{equation}
 by Theorem \ref{independency multiplied by matrix_0}.
From (\ref{preservation of independence 2}) and (\ref{preservation of independence 3}) we get the desired conclusion.
\end{proof}

\bigskip

If $s \ge l$ in the above corollary, then we cannot derive the same conclusion.

\begin{corollary}\label{independency multiplied by matrix_2}
If the conditions $s\le l$ and $\hbox{Rank}A=s$ are replaced by $s\ge l$ and $\hbox{Rank}A= l$ in the setting of Corollary \ref{independency multiplied by matrix_1}, then we have $$ \hbox{ind}_{\mathcal N}\{\varphi_1,\cdots, \varphi_s\} - (s-l)\le \hbox{ind}_{\mathcal N}\{\psi_1, \cdots, \psi_l\}\le \hbox{ind}_{\mathcal N}\{\varphi_1,\cdots, \varphi_s\}.$$
\end{corollary}

\begin{proof}
Since $\hbox{Rank} A =l$ and the entries of $A$ are in $\mathcal N^\infty$, the argument in the proof of Corollary \ref{independency multiplied by matrix_1} shows that there are $l$ rows $a_{i_1}=(a_{i_1 1},\cdots, a_{i_1 l}), \cdots, a_{i_l}=(a_{i_l 1},\cdots, a_{i_l l})$ of $A$ such that $\det \left[
                                                                      \begin{array}{c}
                                                                        a_{i_1} \\
                                                                        \vdots \\
                                                                        a_{i_l} \\
                                                                      \end{array}
                                                                    \right] \ne 0
$ almost everywhere on $\mathbb T$. Without loss of generality, assume that the $\det A_1 \ne 0$, where $A_1$ is the $l\times l$ submatrix function formed with the first $l$ rows of $A$. Write $A= \left[
                             \begin{array}{c}
                               A_1 \\
                               A_2 \\
                             \end{array}
                           \right]
$
and extend $A$ to the square matrix $A^\prime :=\left[
                                \begin{array}{cc}
                                  A_1 & 0 \\
                                  A_2 & I \\
                                \end{array}
                              \right]
$, where, $I$ is the $(s-l)\times (s-l)$ identity matrix, then $\det A^\prime =(\det A_1) \cdot (\det I)= \det A_1 \ne 0$.
Let $\Psi^\prime =\Phi A^\prime=[\psi_1,\cdots,\psi_l\,\cdots ,\psi_s]$, then Theorem \ref{independency multiplied by matrix_0} guarantees $$\hbox{ind}_{\mathcal N}\{\psi_1, \cdots, \psi_s\}=\hbox{ind}_{\mathcal N}\{\varphi_1, \cdots, \varphi_s\},$$ from which, we have $$\hbox{ind}_{\mathcal N}\{\psi_1, \cdots, \psi_l\} \le \hbox{ind}_{\mathcal N}\{\varphi_1, \cdots, \varphi_s\}.$$ Since $\Psi^\prime$ is an extension of $\Phi A= \Psi=[\psi_1,\cdots, \psi_l]$, it is also easy to derive $$\hbox{ind}_{\mathcal N}\{\psi_1, \cdots, \psi_l\}\ge\hbox{ind}_{\mathcal N}\{\psi_1, \cdots, \psi_s\}-(s-l)=\hbox{ind}_{\mathcal N}\{\varphi_1, \cdots, \varphi_s\}-(s-l).$$
\end{proof}

Note in the above corollary that if $s=l$, then the inequalities become equalities, hence, we have Theorem \ref{independency multiplied by matrix_0} in this case.
If the multiplication by $A$ is to the left side of $\Phi$, we have similar results.

\begin{theorem}\label{preservation of independence under left multiplication 2}
Let $A$ be an $m\times m$ matrix-valued function with entries in $\mathcal N^\infty$ such that $\det A \ne 0$ and $\Phi= [\varphi_1, \cdots, \varphi_r] \in L^2_{M_{m \times r}}$. Let $\psi_i$ denote the $i$-th column of $\Psi:= A \Phi$, then $$\hbox{ind}_{\mathcal N}\{\varphi_1, \cdots, \varphi_r\}=\hbox{ind}_{\mathcal N}\{\psi_1, \cdots, \psi_r\}.$$
\end{theorem}
\begin{proof}
  Let $\alpha:= \hbox{ind}_{\mathcal N}\{\varphi_1,\cdots, \varphi_r\}$ and $\beta:=\hbox{ind}_{\mathcal N}\{\psi_1,\cdots,\psi_r\}$.
    If $\alpha=r$, then the statement was proved in Proposition \ref{preservation of independence under left multiplication}. Now suppose $\alpha < r$. Without loss of generality, we may assume that the first $\alpha$ columns $\varphi_1,\cdots,\varphi_\alpha$ are independent mod $\mathcal N$. Note that $$\psi_i= A \varphi_i\hbox{ for each } i=1,\cdots, r
      $$
      and
      $$ [\psi_1,\cdots, \psi_\alpha]=A[\varphi_1,\cdots,\varphi_\alpha].$$
      By Proposition \ref{preservation of independence under left multiplication}, the column vectors $\psi_1,\cdots, \psi_\alpha$ are independent mod $\mathcal N.$ Since, $\alpha$ is the order of any maximal independent subset (mod $\mathcal N$) of $\{\varphi_1,\cdots,\varphi_r\}$, for any $\alpha<i\le r$, the $(\alpha+1)$ column vectors $\varphi_1, \cdots,\varphi_\alpha, \varphi_i$ are not independent mod $\mathcal N$. Thus, again by Proposition \ref{preservation of independence under left multiplication}, we have that $\psi_1,\cdots,\psi_\alpha,\psi_i$ are not independent mod $\mathcal N$ for each $\alpha <i\le r$. Therefore, we conclude that $\{\psi_1,\cdots,\psi_\alpha\}$ is a maximal independent subset (mod $\mathcal N$) of $\{\psi_1,\cdots,\psi_r\}.$ Hence, $\alpha=\beta.$
\end{proof}

\bigskip

\begin{corollary}
  Let $A$ be an $l\times m$ ($l \ge m$) matrix-valued function with entries in $\mathcal N^\infty$ and such that $\hbox{Rank}A=m$. For $\Phi= [\varphi_1 , \cdots , \varphi_r] \in L^2_{M_{m \times r}}$, let $\psi_i$ denote the $i$-th column of $\Psi:= A \Phi$, then $$\hbox{ind}_{\mathcal N}\{\varphi_1, \cdots, \varphi_r\}=\hbox{ind}_{\mathcal N}\{\psi_1, \cdots, \psi_r\}.$$
\end{corollary}

\begin{proof}
  Since we assumed $\hbox{Rank} A =m$, without loss of generality, we may assume that $\det \left[\begin{array}{c}
                 a_{1} \\
                 \vdots \\
                 a_{m}
               \end{array}
  \right] \ne 0$ on a set of positive measure on $\mathbb T$. Set $B:= \left[\begin{array}{c}
                 a_{1} \\
                 \vdots \\
                 a_{m}
               \end{array}
  \right] $, then we can write $A=\left[
                                         \begin{array}{c}
                                           B\\
                                           C \\
                                         \end{array}
                                       \right]
  $, where, $C$ is an $(l-m)\times m$ matrix function. Since $\det B$ is a bounded type function, $\det B\ne 0$ almost everywhere on $\mathbb T.$ Let $\varphi_i^\prime \in L^2_{\mathbb C^l}
  $ be the natural imbedding of $\varphi_i \in L^2_{\mathbb C^m}$ into $L^2_{\mathbb C^l}$, that is, for $\varphi_i=(\varphi_{i1},\cdots,\varphi_{i,m})^t$, let $\varphi^\prime_i=(\varphi_{i1},\cdots, \varphi_{i,m},0,\cdots,0)^t$. Next, define an $l\times l$ matrix $A^\prime := \left[
                        \begin{array}{cc}
                          B & 0 \\
                          C & I \\
                        \end{array}
                      \right]
  $, then
  \begin{equation}\label{imbedding}
 [\psi_1,\cdots, \psi_r]=A^\prime[\varphi_1^\prime,\cdots,\varphi_r^\prime].
  \end{equation}

  Note that $\det A^\prime= \det B \cdot \det I \ne 0$ in the first equation of (\ref{imbedding}). Thus, by Theorem \ref{preservation of independence under left multiplication 2}, we have
  \begin{equation}\label{e2}
  \hbox{ind}_{\mathcal N}\{\psi_1,\cdots, \psi_r\}=\hbox{ind}_{\mathcal N}\{\varphi^\prime_1,\cdots,\varphi^\prime_r\}.
  \end{equation}
 Since $\hbox{ind}_{\mathcal N}\{\varphi^\prime_1,\cdots,\varphi^\prime_r\}=\hbox{ind}_{\mathcal N}\{\varphi_1,\cdots,\varphi_r\}$ is obvious, we finally have
  $$
  \hbox{ind}_{\mathcal N}\{\psi_1,\cdots, \psi_r\}=\hbox{ind}_{\mathcal N}\{\varphi_1,\cdots,\varphi_r\}.
  $$
\end{proof}

\section{Examples of kernels of block Hankel operators}

\bigskip
In this section, some concrete examples of kernels of block Hankel operators will be presented to help understand the results in Section 2.
For a matrix with scalar entries, it is well-known that a maximal set of linearly independent columns and a maximal set of linearly independent rows have the same order. But, for the independency modulo Nevanlinna class, that phenomenon does not happen. Let's begin with a few lemmas.
The following frequently used lemma was proved by J. Long \cite{Lo}.

\begin{lemma}\label{infinite intersection of invariant subspaces}
For a nonconstant inner function $\theta$, $\cap_{n=1}^\infty \theta^n H^2=(0)$.
\end{lemma}

\bigskip

The next lemma shows an example of functions that are not of bounded type.

\begin{lemma}\label{nonbounded typeness sqrt theta}
For a Blaschke factor $B_\alpha:=\frac{z-\alpha}{1-\overline \alpha z}$ ($\alpha \in \mathbb D$), let $B_\alpha^{\frac{1}{2}}$ be an arbitrary function on the unit circle such that $(B_\alpha^{\frac{1}{2}}(z))^2=B_\alpha(z)$. Then $B_\alpha^{\frac{1}{2}}$ is not of bounded type.
\end{lemma}

\begin{proof}
To prove the statement by deriving a contradiction, suppose that $B_\alpha^{\frac{1}{2}}$ is of bounded type. Then since $B_\alpha^{\frac{1}{2}}$ is a unimodular function, we may write
\begin{equation}\label{if bounded type}
B_\alpha^{\frac{1}{2}}=\frac{\theta_1^\prime}{\theta_1}
\end{equation}
 for two inner functions $\theta_1$ and $\theta_1^\prime$ that are mutually coprime.
So we have
 \begin{equation}\label{if bounded type 2}
B_\alpha {\theta_1}^2={\theta_1^\prime}^2.
\end{equation}
Since $\theta_1$ and $\theta_1^\prime$ are mutually coprime, $\theta_1^\prime$ contains $B_\alpha$ as an inner factor. Let $\theta_1^\prime= \theta_2^\prime B_\alpha$ for an inner function $\theta_2^\prime$, then plugging it into (\ref{if bounded type 2}), we have

\begin{equation}\label{if bounded type 3}
{\theta_1}^2=(\theta_2^\prime)^2 B_\alpha.
\end{equation}

Using the same argument as the above, we find that $B_\alpha$ is an inner factor of $\theta_1$. Let $\theta_1= \theta_2 B_\alpha$ for an inner function $\theta_2$ and plug it into (\ref{if bounded type 3}), then we have

\begin{equation}\label{if bounded type 4}
B_\alpha \theta_2^2 ={\theta_2^\prime}^2,
\end{equation}
which is of the same form as (\ref{if bounded type 2}) with $\theta_1$ and $\theta_1^\prime$ replaced by $\theta_2$ and $\theta_2^\prime$, respectively. Repeating this process, we find that $\theta_1, \theta_1^\prime \in (B_\alpha)^n H^2$ for every natural number $n$. But, this is a contradiction because $\cap_{n=1}^\infty (B_\alpha)^nH^2=\{0\}$ by Lemma \ref{infinite intersection of invariant subspaces}. Therefore, we conclude that $B_\alpha^{\frac{1}{2}}$ is not of bounded type. The proof is complete.
\end{proof}

\begin{example}\label{zero kernel of Hankel}
Let $\Phi=\left(
            \begin{array}{cc}
              z^{\frac{1}{2}} & (\frac{z-\frac{1}{2}}{1-\frac{1}{2}z})^{\frac{1}{2}} \\
              0 & 0 \\
            \end{array}
          \right).
$ Then, the two columns of $\Phi$ are independent modulo Nevanlinn class, but, the rows of $\Phi$ are not, where, the row vectors are understood as the column vectors of $\Phi^t$.
\end{example}

\begin{proof}
Let $(f,g)^t \in H^2_{\mathbb C^n}$ be in the kernel of $H_\Phi$. Then we have
$$
z^{\frac{1}{2}}f+(\frac{z-\frac{1}{2}}{1-\frac{1}{2}z})^{\frac{1}{2}}g = h
$$
for some $h \in H^2$. Then by manipulating the equation, we have
\begin{equation}\label{left nonbdd right bdd}
 2(\frac{z-\frac{1}{2}}{1-\frac{1}{2}z})^{\frac{1}{2}}gh=(\frac{z-\frac{1}{2}}{1-\frac{1}{2}z})g^2  +h^2 - zf^2.
\end{equation}

If $gh\ne 0$, then by Lemma \ref{nonbounded typeness sqrt theta}, the left hand side of (\ref{left nonbdd right bdd}) is not a bounded type function, while the right hand side is of bounded type obviously. Thus we have $gh=0$. If $g=0$, then from (\ref{left nonbdd right bdd}), $h^2=zf^2$. Using the same argument as we proved Lemma \ref{nonbounded typeness sqrt theta}, we can prove $f=h=0$. If $h=0$, then plugging it into (\ref{left nonbdd right bdd}) we have
$$(\frac{z-\frac{1}{2}}{1-\frac{1}{2}z})g^2  = zf^2.$$
 Again, using the same arguments as Lemma \ref{nonbounded typeness sqrt theta}, one can verify $f=g=0$. So we have $f=g=0$ in either case, that is, $\ker H_\Phi=(0)$. Therefore, by Proposition \ref{equivalent cond of ind mod N}, the columns of $\Phi$ are independent modulo Nevanlinna class. But, for the row vectors of $\Phi$, which we identify as the column vectors of $\Phi^t$, it is clear that they are not independent modulo Nevanlinna class because one of the vectors is a zero vector.
 \end{proof}

 Actually, it is easy to verify that $\ker H_{\Phi^t}=0\oplus H^2$. Using a matrix inner function we can write $\ker H_{\Phi^t}= \left[
                                                \begin{array}{c}
                                                  0 \\
                                                  1 \\
                                                \end{array}
                                              \right]H^2$,
where we note that  $\left[
                                                \begin{array}{c}
                                                  0 \\
                                                  1 \\
                                                \end{array}
                                              \right]
$ is a matrix inner function of size $2\times 1$.
 A similar, but, slightly more manipulated method will prove that for three distinct values $\alpha_1, \alpha _2, \alpha_3 \in \mathbb D$, $\{ B^\frac{1}{2}_{\alpha_1},B^\frac{1}{2}_{\alpha_2},B^\frac{1}{2}_{\alpha_3}\}$ is independent modulo Nevanlinna class. The following question seems plausible, but it seems not easy to answer it in the affirmative.
\begin{question}
 Is the set $\{B^\frac{1}{2}_{\alpha_1},\cdots, B^\frac{1}{2}_{\alpha_r}\}$ independent modulo Nevanlinna class for arbitrary $r$ distinct complex numbers $\alpha_i \in \mathbb D?$
\end{question}

On the other hand, it is easy to form a set $\{\varphi_1, \cdots, \varphi_r\}$ of independent vectors(mod $\mathcal N$) in $L^\infty_{\mathbb C^n}$ when $r\le n$. Indeed, if $\varphi_i=(\varphi_{i1},\cdots, \varphi_{in})^t$ such that $\varphi_{ij} \in \mathcal N^\infty$ if $i\ne j$ and $\varphi_{ii} \in L^\infty \setminus\mathcal N^\infty$, then $\varphi_1,\cdots,\varphi_r$ are easily verified to be independent modulo Nevanlinna class.
Then, is there a set of scalar-valued functions $\{a_1, \cdots, a_r\} \subset L^\infty$ that is independent modulo Nevanlinna class for arbitrary natural number $r$? The answer is `yes' and it will be shown in the next section.

\bigskip

The following result was shown in \cite{GHR}.

\begin{proposition}
\cite{GHR}\label{kerofHankel with square inner adjoint0}
   Let $\Phi\in L^\infty_{M_n}$, then $\ker H_\Phi= \Theta H^2_{\mathbb C^n}$ for some square matrix inner function $\Theta\in H^\infty_{M_n}$ if and only if $\Phi= A \Theta^*$, where, $A\in H^\infty_{M_n}$ and $A$ and $\Theta$ don't have a  common nonconstant matrix inner function.
\end{proposition}

 For a scalar inner function $\theta$, $\ker H_{\overline \theta}= \theta H^2$ is easily verified. For matrix inner functions, the following is a special case of the above proposition.

   \begin{corollary}
   \label{kerofHankel with square inner adjoint}
   If $\Theta$ is an $n\times n$ matrix-valued inner function, then $\ker H_{\Theta^*}= \Theta H^2_{\mathbb C^n}$.
   \end{corollary}
 But, this seemingly natural phenomenon does not always happen when the inner function $\Theta$ is not square.
 For example, for nonconstant scalar inner functions $\theta_1$ and $\theta_2$, consider a $2\times 1$ matrix inner function $\Theta:=[\frac{1}{\sqrt2}\theta_1 , \frac{1}{\sqrt2}\theta_2]^t$. Since each entry of $\Theta^*=[\frac{1}{\sqrt2}\overline \theta_1 , \frac{1}{\sqrt2}\overline \theta_2]$ is of bounded type, $\hbox{ind}_{\mathcal N}\{\frac{1}{\sqrt2}\overline \theta_1, \frac{1}{\sqrt2}\overline \theta_2\}$ is 0. Thus, by Theorem \ref{ker of block Hankel}, there exists a $2\times2$ matrix inner function $\Theta_1$ such that $\ker H_{\Theta^*}=\Theta_1 H^2_{\mathbb C^2}$. Obviously, $\Theta_1 \ne \Theta$ because they have different sizes.
 Now let's consider the following natural questions.

\medskip
\begin{equation}\label{questions}
\begin{array}{l}
        \hbox{Q1. What is the $2\times 2$ matrix $\Theta_1$ concretely?} \\
       \hbox{Q2. For what matrix function $\Phi$, $\ker H_\Phi = \Theta H^2 = [\frac{1}{\sqrt2}\theta_1 , \frac{1}{\sqrt2}\theta_2]^t H^2$?}
      \end{array}
\end{equation}

\medskip
To give answers to these questions, we need to consider the greatest common inner divisor(GCD) of $\theta_1$ and $\theta_2$. Let $\theta_1=\theta \theta_1^\prime$ and $\theta_2=\theta \theta_2^\prime$, where $\theta$ is the GCD of $\theta_1$ and $\theta_2$ in the sense that $\theta_1^\prime$ and $\theta_2^\prime$ have no nonconstant common inner devisor.
First note that $\ker H_{[\frac{1}{\sqrt 2}\overline \theta_1 , \frac{1}{\sqrt 2}\overline \theta_2]} = \ker H_{[\overline \theta_1 , \overline \theta_2]} $ and let $\left(
                                        \begin{array}{c}
                                          f \\
                                          g \\
                                        \end{array}
                                      \right) \in \ker H_{[\overline \theta_1 , \overline \theta_2]}, $ then

                                      $$[\overline \theta_1 , \overline \theta_2]\left(
                                        \begin{array}{c}
                                          f \\
                                          g \\
                                        \end{array}
                                      \right)= \overline \theta_1 f +\overline \theta_2 g =h $$
for some $h\in H^2$. Multiplying this equation by $\theta \theta_1^\prime \theta_2^\prime$, we have

\begin{equation}\label{ker of 1 by 2}
\theta_2^\prime f +\theta_1^\prime g =\theta \theta_1^\prime \theta_2^\prime h.
\end{equation}
Since $\theta_1^\prime$ and $\theta_2^\prime$ are coprime, $f=\theta_1^\prime f^\prime$ and $g= \theta_2^\prime g^\prime$ for some $f^\prime,g^\prime \in H^2$.
Plugging it into (\ref{ker of 1 by 2}), we have

\begin{equation}
f^\prime + g^\prime =\theta h.
\end{equation}

Therefore, $$\begin{array}{ll}
             \left(
                                        \begin{array}{c}
                                          f \\
                                          g \\
                                        \end{array}
                                      \right)& = \left(
                                        \begin{array}{c}
                                          \theta_1^\prime f^\prime \\
                                          \theta_2^\prime (\theta h - f^\prime) \\
                                        \end{array}
                                      \right)\\
                                      & = \left[
                                                 \begin{array}{cc}
                                                   \theta_1^\prime & 0 \\
                                                   0 & \theta_2^\prime\\
                                                 \end{array}
                                               \right]
                                      \left(\left(
                                        \begin{array}{c}
                                          \frac{1}{2}\theta h \\
                                          \frac{1}{2}\theta h \\
                                        \end{array}
                                      \right)
                                      +
                                      \left(
                                        \begin{array}{c}
                                          f^\prime -  \frac{1}{2}\theta h\\
                                          - f^\prime +  \frac{1}{2}\theta h \\
                                        \end{array}
                                      \right)\right) \\
               &  = \left[
                                                 \begin{array}{cc}
                                                   \theta_1^\prime & 0 \\
                                                   0 & \theta_2^\prime\\
                                                 \end{array}
                                               \right]\left[
                                                 \begin{array}{cc}
                                                   \frac{1}{2}\theta- \frac{1}{2} & \frac{1}{2}\theta + \frac{1}{2} \\
                                                  \frac{1}{2}\theta+ \frac{1}{2} & \frac{1}{2}\theta - \frac{1}{2}\\
                                                 \end{array}
                                               \right]\left(
                                                        \begin{array}{c}
                                                          f_1 \\
                                                          g_1\\
                                                        \end{array}
                                                      \right)
                                               ,
            \end{array}
$$

                                               where, $f_1= -f^\prime +\frac{1}{2}\theta h +\frac{h}{2}$ and $g_1=f^\prime -\frac{1}{2}\theta h +\frac{h}{2}$.

                                              A direct calculation shows that $\left[
                                                 \begin{array}{cc}
                                                   \frac{1}{2}\theta- \frac{1}{2} & \frac{1}{2}\theta + \frac{1}{2} \\
                                                  \frac{1}{2}\theta+ \frac{1}{2} & \frac{1}{2}\theta - \frac{1}{2}\\
                                                 \end{array}
                                               \right]$ is a matrix-valued inner function.

                                                Therefore,
                                                $$\ker H_{[\overline \theta_1 , \overline \theta_2]}\subset\Theta_1 H^2_{\mathbb C^2},
                                                $$
                                                  where,

\begin{equation}\label{theta1}
\Theta_1= \left[
                                                 \begin{array}{cc}
                                                   \theta_1^\prime & 0 \\
                                                   0 & \theta_2^\prime\\
                                                 \end{array}
                                               \right]\left[
                                                 \begin{array}{cc}
                                                   \frac{1}{2}\theta- \frac{1}{2} & \frac{1}{2}\theta + \frac{1}{2} \\
                                                  \frac{1}{2}\theta+ \frac{1}{2} & \frac{1}{2}\theta - \frac{1}{2}\\
                                                 \end{array}
                                               \right].
\end{equation}

 The opposite inclusion($\ker H_{[\overline \theta_1 , \overline \theta_2]}\subset\Theta_1 H^2_{\mathbb C^2}$) is easily proved because for $f,g \in H^2$, $$
                                               \begin{array}{ll}
                                                \Theta^* \Theta_1\left(
                                                        \begin{array}{c}
                                                          f \\
                                                          g \\
                                                        \end{array}
                                                      \right) & =\ [\overline \theta_1 , \overline \theta_2]\left[
                                                 \begin{array}{cc}
                                                   \theta_1^\prime & 0 \\
                                                   0 & \theta_2^\prime\\
                                                 \end{array}
                                               \right]\left[
                                                 \begin{array}{cc}
                                                   \frac{1}{2}\theta- \frac{1}{2} & \frac{1}{2}\theta + \frac{1}{2} \\
                                                  \frac{1}{2}\theta+ \frac{1}{2} & \frac{1}{2}\theta - \frac{1}{2}\\
                                                 \end{array}
                                               \right]\left(
                                                        \begin{array}{c}
                                                          f \\
                                                          g \\
                                                        \end{array}
                                                      \right)\\
                                                       & =[\overline \theta_1 , \overline \theta_2]
                                               \left(
                                                        \begin{array}{c}
                                                          \theta_1^\prime(\frac{1}{2}\theta(f+g)-\frac{1}{2}(f-g)) \\
                                                           \theta_2^\prime(\frac{1}{2}\theta(f+g)+\frac{1}{2}(f-g)) \\
                                                        \end{array}
                                                      \right) \\
                                                  & = f+g \in H^2.
                                               \end{array}
                                                $$
Thus, we conclude that
\begin{equation}\label{theta2}
 \ker H_{[\frac{1}{\sqrt 2}\overline {\theta_1},\frac{1}{\sqrt 2}\overline {\theta_1}]}= \Theta_1 H^2_{\mathbb C^2}, \hbox{ where, } \Theta_1= \left[
                                                 \begin{array}{cc}
                                                   \theta_1^\prime & 0 \\
                                                   0 & \theta_2^\prime\\
                                                 \end{array}
                                               \right]\left[
                                                 \begin{array}{cc}
                                                   \frac{1}{2}\theta- \frac{1}{2} & \frac{1}{2}\theta + \frac{1}{2} \\
                                                  \frac{1}{2}\theta+ \frac{1}{2} & \frac{1}{2}\theta - \frac{1}{2}\\
                                                 \end{array}
                                               \right].
\end{equation}

In fact, the matrix inner function
$$\left[
                                                 \begin{array}{cc}
                                                   \frac{1}{2}\theta- \frac{1}{2} & \frac{1}{2}\theta + \frac{1}{2} \\
                                                  \frac{1}{2}\theta+ \frac{1}{2} & \frac{1}{2}\theta - \frac{1}{2}
                                                  \end{array}\right] = \left[
                                                 \begin{array}{cc}
                                                   \frac{1}{2}\theta & \frac{1}{2}\theta \\
                                                  \frac{1}{2}\theta & \frac{1}{2}\theta
                                                  \end{array}\right]+
                                                  \left[
                                                 \begin{array}{cc}
                                                   - \frac{1}{2} & \frac{1}{2} \\
                                                  \frac{1}{2} & - \frac{1}{2}
                                                  \end{array}\right]$$ can be analyzed as $I_\theta P_N + P_{N^\perp}$, where, $N$ is a subspace of $\mathbb C^2$ defined by $N:=\{(\alpha,\alpha): \alpha \in \mathbb C\}$ and $P_N$ and $P_{N^\perp}$ are the orthogonal projection of $\mathbb C^2$ onto $N$ and $N^\perp:= \mathbb C^2\ominus N$, respectively.
\bigskip

For an answer to the second question in (\ref{questions}), consider any function $a \in {L^\infty}$ which is not of bounded type and let $\Phi:=\left[
\begin{array}{cc}
a \overline {\theta_1} & -a \overline{\theta_2} \\
\overline\theta_1 & 0 \\
 \end{array}
 \right].
$
Then $(f,g)^t \in \ker H_\Phi$ if
\begin{equation}\label{ker 2 by 2}
\left[
\begin{array}{cc}
 a \overline {\theta_1} & - a \overline{\theta_2} \\
\overline\theta_1 & 0 \\
 \end{array}
 \right]
\left(
  \begin{array}{c}
    f\\
    g \\
  \end{array}
\right)=\left(
          \begin{array}{c}
             a(\overline{\theta_1^\prime} f -\overline {\theta_2^\prime}g) \\
          \overline {\theta_1} f \\
          \end{array}
        \right)
 = \left(
  \begin{array}{c}
    h\\
    k \\
  \end{array}
\right)
\end{equation}
for some $h,k\in H^2$. Since $a$ is not of bounded type, we have $\overline{\theta_1} f-\overline {\theta_2}g=h=0$ and $\overline \theta_1f= k$. Thus,
 $$\left(
                                         \begin{array}{c}
                                           f \\
                                           g \\
                                         \end{array}
                                       \right)=
                                       \left(
                                         \begin{array}{c}
                                           \theta_1 k \\
                                           \theta_2 k \\
                                         \end{array}
                                       \right) \in
                                       \left[
                                         \begin{array}{c}
                                           \theta_1  \\
                                           \theta_2  \\
                                         \end{array}
                                       \right]H^2,
$$
which implies $\ker H_\Phi \subset \Theta H^2$. The opposite inclusion is easily verified, hence
$$\ker H_\Phi = \Theta H^2, \hbox{ where, }\Theta = \left[
                                                      \begin{array}{c}
                                                        \frac{1}{\sqrt 2}\theta_1 \\
                                                         \frac{1}{\sqrt 2}\theta_2\\
                                                      \end{array}
                                                    \right]
.$$
Is it possible to have $\Theta H^2= \ker H_\Psi$ for some $(1\times 2)$ matrix function $\Psi$? If we set $\Psi:= [a, \overline \theta_2(1- a\theta_2)]$ for a nonbounded type function $ a \in L^\infty\setminus \mathcal N^\infty$, then
a similar argument as the above shows $\ker H_\Psi= \left[
                                                      \begin{array}{c}
                                                        \frac{1}{\sqrt 2}\theta_1 \\
                                                         \frac{1}{\sqrt 2}\theta_2\\
                                                      \end{array}
                                                    \right]H^2$.

In fact, we have a more general result, namely, Theorem \ref{thm rel bet n by 1 inner and symbol}. In its proof, we use the existence of sets of independent scalar-valued $L^2$-functions mod $\mathcal N$ of arbitrary finite order, which will be shown in the next section. Note that if $k_1,\cdots,k_n \in H^\infty$ and $|k_1|^2+\cdots+|k_n|^2=1$, then $\Theta:=(k_1,\cdots,k_n)^t$ is an $(n\times 1)$ matrix inner function.
\begin{theorem}\label{thm rel bet n by 1 inner and symbol}
For an $n \times 1$ matrix inner function $\Theta=(k_1,
\cdots, k_n)^t$, there exists a $(1\times n)$ matrix function $\Phi$ such that $\Theta H^2 = \ker H_\Phi$.
\end{theorem}

\begin{proof}
Since $k_i$ are analytic functions and $\sum_{i=1}^n |k_i|^2=1$, at least one $k_i$ is a nonzero function.
Without loss of generality, we may assume $k_n\ne0$. ($k_n(\zeta)\ne 0$ a.e. on $\mathbb T$ since $k_n \in H^\infty$)
Let $\{ a_1,\cdots, a_{n-1}\} \subset L^\infty$ be an independent set modulo Nevanlinna class and define a $1\times n$ matrix function $$\Phi=[a_1,\cdots ,a_{n-1}, \frac{1}{k_n}(1-\sum_{i=1}^{n-1}  a_ik_i)].$$ Then
$$
\begin{array}{ll}
  (f_1, \cdots, f_n)^t \in \ker H_\Phi & \Longleftrightarrow \sum_{i=1}^{n-1} a_i f_i +\frac{1-\sum_{i=1}^{n-1}  a_ik_i}{k_n}f_n \in H^2\\
   & \Longleftrightarrow \sum_{i=1}^{n-1}a_i(f_i - \frac{k_i}{k_n}f_n) +\frac{1}{k_n}f_n \in H^2.
\end{array}
$$
Independence of $\{ a_1, \cdots, a_{n-1}\}$ implies $f_i - \frac{k_i}{k_n}f_n=0$ and $\frac{1}{k_n}f_n = h$ for some $h\in H^2$. Therefore $f_i= k_i h$ for each $i=1, \cdots, n$, that is, $(f_1, \cdots, f_n)^t \in \Theta H^2$.
\end{proof}
\bigskip
\\
Theorem \ref{thm rel bet n by 1 inner and symbol} naturally leads us to the following question :

\medskip
For an arbitrary $n\times r$ inner function $\Theta$, does there exist a $1 \times n$ matrix function $\Phi$ such that $\Theta H^2_{\mathbb C^r}= \ker H_\Phi$ for some?

\medskip
The answer to the above question is negative. For a $2\times 2$ matrix inner function $I_z=\left[
                                                \begin{array}{cc}
                                                  z & 0 \\
                                                  0 & z \\
                                                \end{array}
                                              \right]
$, it can be shown that $I_z H^2_{\mathbb C^2}$ cannot be represented as the kernel of a Hankel operator $H_\Phi$ for a $1\times 2$ matrix function $\Phi$.
To derive a contradiction, assume $\ker H_{\Phi} = I_z H^2_{\mathbb C^2}$ for $\Phi =[a_1, a_2]$, where $a_1, a_2 \in L^2$ are scalar-valued $L^2$ functions. Since $(z,0)^t,(0,z)^t \in  I_z H^2_{\mathbb C^2}=\ker H_\Phi$, we have $$
a_1 = \alpha_1 \overline z + h_1 \hbox{ and } a_2 = \alpha_2 \overline z + h_2,
$$ for some complex numbers $\alpha_i\in \mathbb C$ and analytic functions $h_i \in H^2$($i=1,2$). Since $\ker H_\Phi$ does not contain $(1,0)^t$ and $(0,1)^t$, we find $\alpha_1\ne 0$ for $i=1,2.$ It is easy to verify
$$\ker H_{[{\alpha_1 \overline z} +  h_1,{\alpha_2 \overline z} +  h_2]}= \{(f, g)^t \in H^2_{\mathbb C^2}: {\alpha_1}f(0)+{\alpha_2}g(0)=0 \}.$$
This set obviously contains, but, is not equal to $I_z H^2_{\mathbb C^2}$.

\section{Application to invariant subspaces for $S$ and $S^*$}

As an application of the main result, we will discuss the shape of invariant subspaces of $H^2_{\mathbb C^n}$ for $S$ and $S^*$ generated by finite elements.
For
$A=\{f_i
\in H^2_{\mathbb C ^n}| i=1,2, \cdots,s \}$, define $E_A$ and $E_A^*$ by
 $$E_A :=\bigvee_{k\in \mathbb{N}\cup\{0\}, f_i \in A} {S}^{k} f_i \hbox{ and } E^*_A :=\bigvee_{k\in \mathbb{N}\cup\{0\}, f_i \in A} {S^*}^{k} f_i$$
Clearly, $E_A$ and $E_A^*$ are invariant subspaces of $H^2_{\mathbb C^n}$ for the shift operator $S$ and the backward shift $S^*$, respectively.
Therefore, by Beurling-Lax-Halmos Theorem, there exist natural numbers $m,m^\prime (\le  n)$ and matrix inner functions $\Theta \in H^\infty_{M_{n\times m}}, \Theta^\prime \in H^\infty_{M_{n\times m^\prime}}$ such that $E_A=\Theta H^2_{\mathbb C^m}$ and $E^*_A = H^2_{\mathbb C^n}\ominus \Theta^\prime H^2_{\mathbb C^{m^\prime}}$.
The set $A$ will be called a generating set of $E_A$ for $S$, or, a generating set of $E^*_A$ for $S^*$. Our concern is how the numbers $m$ and $m^\prime$ is determined, or, in other words, how the sizes of inner matrices $\Theta$ and $\Theta^\prime$ are decided.
To figure out how the sizes of the matrix inner functions $\Theta$ and $\Theta^\prime$ are determined, the following lemma is useful.
\begin{lemma}\label{backshift_invsub_genby_finvec0}
  Let $A=\{ f_1, f_2, \cdots, f_s\} \subset H^2_{\mathbb C^n}$ and define an $n\times s$ matrix  $F:=[f_1,f_2 , \cdots ,f_s]$. Then
$E^*_A = H^2_{\mathbb C^n} \ominus \ker H_{\overline z F^*}$.
\end{lemma}
\begin{proof}
  Write $f_i=\left(
         \begin{array}{c}
           f_{i1} \\
           \vdots \\
           f_{in} \\
         \end{array}
       \right)$
and suppose $h= \left(
         \begin{array}{c}
           h_{1} \\
           \vdots \\
           h_{n} \\
         \end{array}
       \right) \in  H^2_{\mathbb C^n}\ominus E^*_A$.
Then for each $k \in \mathbb{N}\cup \{0\}$, we have
\begin{equation}\label{back shift inv subsp}
\begin{array}{ll}
  0 & =\langle h, S^{*k} f_i \rangle \\
     &  =\langle S^k h, f_i \rangle\\
     & =\langle  \left(
         \begin{array}{c}
           z^k h_{1} \\
           \vdots \\
           z^k h_{n} \\
         \end{array}
       \right) , \left(
         \begin{array}{c}
           f_{i1} \\
           \vdots \\
           f_{in} \\
         \end{array}
       \right) \rangle \\
     & = \int_{\mathbb T} z^k(h_1 \overline{f_{i1}}+\cdots+ h_n \overline{f_{in}})dm.
\end{array}
\end{equation}

Since (\ref{back shift inv subsp}) implies that the 0-th and negative Fourier coefficients of $(h_1 \overline{f_{i1}}+\cdots+ h_n \overline{f_{in}})$ are all zero,
 $$(h_1 \overline{f_{i1}}+\cdots+ h_n \overline{f_{in}}) \in zH^1, \hbox{ or, equivalently, }  (h_1 \overline{zf_{i1}}+\cdots+ h_n \overline{zf_{in}}) \in H^1.$$
Since $f_i$ is an arbitrary element in $A$, we conclude that
$h \in \ker H_{\overline z F^*}$.
Therefore, $$H^2_{\mathbb C^n}\ominus E^*_A \subset \ker H_{\overline z F^*}.$$
It is easy to verify that the above argument can be reversed, i.e., $h\in \ker H_{\overline z F^*}$ implies $h \perp S^{*k}f_i$ for each $k\in \mathbb N \cup \{0\}$ and $i= 1,\cdots, s$. Thus, we have $$H^2_{\mathbb C^n}\ominus E^*_A \supset \ker H_{\overline z F^*},$$ hence, $H^2_{\mathbb C^n}\ominus E^*_A = \ker H_{\overline z F^*},$ or, equivalently,
$ E^*_A= H^2_{\mathbb C^n}\ominus \ker H_{\overline z F^*}$ as desired.
\end{proof}

\bigskip

The following theorem is immediate from Lemma \ref{backshift_invsub_genby_finvec0} and Theorem \ref{ker of block Hankel}.

\begin{theorem}\label{backshift_invsub_genby_finvec1}
For a finite subset $A=\{ f_1, f_2, \cdots, f_s\}$ of $ H^2_{\mathbb C^n}$, define an $n\times s$ matrix  $F:=[f_1 ,f_2 , \cdots ,f_s]$ and let $\overline g_i \in \overline{H^2_{\mathbb C^s}}$ be the $i$-th column of $F^*$. Set $\alpha := \hbox{ind}_{\mathcal N}\{\overline g_1,\cdots, \overline g_n\}$, then
$$E^*_A = H^2_{\mathbb C^n}\ominus \Theta H^2_{\mathbb C^{n-\alpha}}$$ for some $n \times(n- \alpha)$ matrix inner function $\Theta$.
\end{theorem}

\bigskip

 For an $n\times n$ matrix-valued inner function $\Theta$, let $f_i$ be the $i$-th column of $\Theta$. Then it is easy to verify that $\{S^*f_i\,\,|\,\,i=1,\cdots, n\}$ is a generating set of $H^2\ominus \Theta H^2_{\mathbb C^n}$ for $S^*$. (See \cite{Ni}, page 41) But, this observation may fail if $\Theta$ is a nonsquare matrix inner function. \begin{example}
   Consider a $2\times 1$ matrix inner function $\Theta= [\frac{1}{\sqrt{2}} z, \frac{1}{\sqrt{2}} z]^t$, then the invariant subspace for $S^*$ generated by the single element $S^*(\frac{1}{\sqrt{2}} z,\frac{1}{\sqrt{2}} z)^t=(\frac{1}{\sqrt 2},\frac{1}{\sqrt 2})^t$ is $H^2_{\mathbb C^2}\ominus \ker H_{[\overline z ,\overline z]}$ by Lemma \ref{backshift_invsub_genby_finvec0}. But, recalling the shape of $\ker H_{[\theta_1, \theta_2]}$ in Section 4( see (\ref{theta2})), we find
  $$\ker H_{[\overline z ,\overline z]}=
 \left[
   \begin{array}{cc}
     \frac{1}{2}z-\frac{1}{2} & \frac{1}{2}z+\frac{1}{2} \\
     \frac{1}{2}z+\frac{1}{2} & \frac{1}{2}z-\frac{1}{2}
   \end{array}
 \right]
 H^2_{\mathbb C^2}
 $$
 Therefore,
 $$E^*_{S^*(\frac{1}{\sqrt{2}} z,\frac{1}{\sqrt{2}} z)^t}= H^2_{\mathbb C^2} \ominus  \left[
   \begin{array}{cc}
     \frac{1}{2}z-\frac{1}{2} & \frac{1}{2}z+\frac{1}{2} \\
     \frac{1}{2}z+\frac{1}{2} & \frac{1}{2}z-\frac{1}{2}
   \end{array}
 \right]
 H^2_{\mathbb C^2} \ne  H^2_{\mathbb C^2}\ominus [\frac{1}{\sqrt{2}} z,\frac{1}{\sqrt{2}} z]^t H^2.$$
 \end{example}

 \medskip
 Generally, it is known that any invariant subspace of $H^2_{\mathbb C^n}$ for the backward shift $S^*$ has a generating set of order less than or equal to $n$.
In fact, more is known.

\begin{theorem}\label{gen_set_thm_for_backshift}(See \cite {Ni}, page 41)
Let $\Theta$ be a matrix inner function. If $\Theta \in H^\infty_{M_{m\times m}}$, then the invariant subspace $H^2_{\mathbb C^m}\ominus \Theta H^2_{\mathbb C^m}$ for the backward shift $S^*$ has a generating set of order $m$. If $\Theta \in H^\infty_{M_{n\times m}}$ with $m <n$, then $H^2_{\mathbb C^n}\ominus \Theta H^2_{\mathbb C^m}$ has a generating set of order $m+1$.
\end{theorem}

\begin{corollary}\label{gen set cor}
  For any matrix inner function $\Theta \in H^\infty_{M_{n\times m}}$, there exists a matrix function $\Phi \in L^2_{r\times n}$ such that $\Theta H^2_{\mathbb C^m}= \ker H_\Phi$ and $r\le n$.
\end{corollary}

\begin{proof}
Using Theorem \ref{gen_set_thm_for_backshift}, there exists a finite set $A=\{g_1,\cdots,g_r\}\subset H^2_{\mathbb C^n}$ ($r\le n$) such that $E^*_A= H^2_{\mathbb C^n}\ominus \Theta H^2_{\mathbb C^m}$. Let $\Phi= [\overline {zg_1},\cdots, \overline {zg_r}]^t \in L^2_{M_{r\times n}}$, then by Lemma \ref{backshift_invsub_genby_finvec0}, $\ker H_{\Phi}= \Theta H^2_{\mathbb C^m}$.
\end{proof}

\bigskip
For a Hilbert space $\mathcal H$ and a bounded linear operator $T$ on $\mathcal H$, an element $f\in \mathcal{H}$ is said to be a \textit{cyclic vector} for $T$ if
$$\bigvee_{k=0,1,\cdots} T^{k}f=\mathcal H.$$
It is known that the set of cyclic vectors for $S^*$ in $H^2_{\mathbb C^n}$ is nonempty, even further, it is a dense subset of $H^2_{\mathbb C^n}$(\cite{Ni}, page 41).
The following proposition shows the existence of a set of scalar-valued functions of arbitrary order that is independent modulo Nevanlinna class.

\begin{proposition}\label{cyclic vector and independency}
$(f_1,\cdots, f_n)^t \in H^2_{\mathbb C^n}$ is a cyclic vector for the backward shift $S^*$ (i.e., $\bigvee_{k=0,1,\cdots} S^{*k} (f_1,\cdots, f_n)^t = H^2_{\mathbb C^n}$) if and only if the set $\{\overline f_1, \cdots, \overline f_n\}\subset L^2$ is independent modulo Nevanlinna class.
\end{proposition}

\begin{proof}
  Let $f=(f_1,\cdots, f_n )^t\in H^\infty_{\mathbb C^n}$ and $E^*_f:=\bigvee_{k=0,1,\cdots} S^{*k}f$. Theorem \ref{backshift_invsub_genby_finvec0} shows that a vector $g \in H^2_{\mathbb C^n}$ belongs to $H^2_{\mathbb C^n}\ominus E^*_f $ if and only if $g \in \ker H_{[\overline{zf_1},\cdots, \overline{zf_n} ]}$. Therefore, by Theorem \ref{ker of block Hankel}, $E^*_f= H^2_{\mathbb C^n}$ if and only if $\{\overline{zf_1},\cdots, \overline{zf_n} \}$ is independent modulo Nevanlinna class.  It is easy that independence of $\{\overline{zf_1},\cdots, \overline{zf_n} \}$ is equivalent to independence of $\{\overline{f_1},\cdots, \overline{f_n} \}$(see (\ref{independency multiplied by diagonal matrix})). The proof is complete.
\end{proof}

\bigskip
Since the set of cyclic vectors for $S^*$ in $H^2_{\mathbb C^n}$ is nonempty as was mentioned before, Proposition \ref{cyclic vector and independency} immediately gives
\begin{corollary}
  For an arbitrary natural number $ n $, there exist $n$ (scalar-valued) functions in $L^2$ that are independent modulo Nevanlinna class.
\end{corollary}

\bigskip

The next theorem is the main result on finitely generated shift-invariant subspaces of $H^2_{\mathbb C^n}$. The proof of the theorem contains a subtle difficulty.

\begin{theorem}\label{shiftinvarsubsgeneratedby}
    For a finite set of vectors $A=\{f_1,\cdots f_s\}\subset H^2_{\mathbb C^n}$, define a matrix function $F:=[f_1,\cdots, f_s]\in H^2_{M_{n\times s}}$ and let $m:=\hbox{Rank}F$. Then $E_A = \Theta H^2_{\mathbb C^m}$ for some $n\times m$ matrix inner function $\Theta$.
\end{theorem}

\begin{proof}
Since $E_A$ is a shift invariant subspace, let $E_A = \Theta H^2_{\mathbb C^{m^\prime}}$ for some natural number $m^\prime$ and an $n\times m^\prime$ matrix inner function $\Theta$. What needs be proved is $m^\prime = m$.

Since $f_i \in E_A$ for each $i=1,\cdots,s$, using Lemma \ref{size of inner function}, we have $m^\prime \ge \hbox{Rank} F=m$.

To show the opposite inequality $m^\prime \le m$, Lemma \ref{size of inner function} will be used again.
Let $G:=[g_1, \cdots ,g_r]$ be an $n\times r$ matrix function, where, $r$ is an arbitrary natural number and $g_i\in E_A$ for each $i$. Referring to Lemma \ref{size of inner function}, it suffices to show that $\hbox{Rank}G \le \hbox{Rank}F=m$. If $\min\{n,r\} \le m$, then $\hbox{Rank}G \le m$ is trivially satisfied. Thus, now we assume $n,r \ge m +1$.
  By the definition of $E_A$, we notice that
  $$E_A = \hbox{cl} \{ F \cdot p \, \, | \, \,p \hbox{ is an analytic polynomial in }  \mathcal P _+({\mathbb C^s})\}.
  $$
 Since $g_i \in \hbox{cl} F\mathcal P_+(\mathbb C^s)$, there exists a sequence of polynomials $\{p_{ik}\}_{k=1}^\infty \subset \mathcal P_+(\mathbb C^s)$, such that $\lim_{k\rightarrow \infty}F\cdot p_{ik}=g_i$ for each $i=1,\cdots,r$, where, the convergence is in the $L^2_{\mathbb C^n}$-norm. Note that $L^2$-convergence of a sequence of vector-valued functions is equivalent to $L^2$-convergence of each coordinate functions. Therefore, in the sense of $L^2$-convergence of each entry function, $$G= \lim_{k\rightarrow \infty}F\cdot[p_{1k},\cdots, p_{rk}].$$ Write $P_k:= [p_{1k},\cdots, p_{rk}]$ for convenience.
  Note that
  $$\hbox{rank}( F(\zeta)\cdot P_k(\zeta)) \le \hbox{rank} F(\zeta).$$
    Let $G^\prime$ be an arbitrary $(m+1)\times(m+1)$ square submatrix of $G$ obtained by removing some columns and rows of $G$. Define $G_k^\prime = F^\prime P_k^\prime$, where $F^\prime$ is the submatrix of $F$ obtained by removing the corresponding rows as were removed from $G$ to form $ G^\prime$ and $P_k^\prime$ is the submatrix of $P_k$ obtained by removing the corresponding columns as were removed from $G$ to form $G^\prime$. Clearly, $$G^\prime= \lim_{k\rightarrow \infty}F^\prime P^\prime_k=\lim_{k\rightarrow \infty}G_k^\prime$$ in the entrywise $L^2$-norm. Observe
    \begin{equation}\label{rank of submatrix}
      \hbox{rank}G^\prime_k(\zeta)= \hbox{rank}F^\prime(\zeta)\le \hbox{rank} F(\zeta) \le m
    \end{equation}
    for almost all $\zeta \in \mathbb T$, where the last inequality is from our assumption $\hbox{Rank}F=m$. Note that, since each entry function of $G^\prime$ and $G^\prime_k$ belongs to $L^2(\mathbb T)$, there exist a positive number $q$ such that $\det G^\prime , \det G^\prime_k \in L^q(\mathbb T)$ and $\det G^\prime_k$ converges to $\det G^\prime$ in $L^q$-norm.
     But, since $G^\prime_k$ is an $(m+1)\times(m+1)$ square matrix satisfying (\ref{rank of submatrix}), we have $\det G^\prime_k(\zeta)=0$ almost everywhere on $\mathbb T$, that is, $\det G^\prime =0$, a zero function. Since the only $L^q$-limit of the sequence of zero functions is the zero function itself, we conclude that $\det G^\prime=0$, a zero function. Recall that $G^\prime$ is an arbitrary $(m+1)\times (m+1)$ square submatrix of $G$. Therefore, we conclude that $\hbox{rank} G(\zeta) \le m$ almost everywhere on $\mathbb T$, that is, $\hbox{Rank} G \le m$. Since $G$ is an arbitrary finite matrix whose columns are in $E_A= \Theta H^2_{\mathbb C^{m^\prime}}$, using Lemma \ref{size of inner function}, we have $m^\prime \le m$.
 The proof is complete.
\end{proof}

\bigskip
\begin{remark}(Inner-outer factorization of matrix-valued analytic functions)\\
 For a matrix-valued analytic function $F \in H^2_{M_{n\times m}}$, let $F = \Theta G$ be the inner-outer factorization. Theorem \ref{shiftinvarsubsgeneratedby} can be used to find the size of the matrix-valued functions $\Theta$ and $G$. Suppose $\Theta \in H^\infty_{M_{n\times r}}$ and $G\in H^2_{M_{r\times m}}$ for a natural number $r$, then
 $$
 FH^2_{\mathbb C^m}= \Theta G H^2_{\mathbb C^m}= \Theta H^2_{\mathbb C^r}$$
  since $G$ is an outer function. Since $F H^2_{\mathbb C^m}$ is the invariant subspace generated by the columns of $F$, Theorem \ref{shiftinvarsubsgeneratedby} implies $r=\hbox{Rank}F$.
\end{remark}

\bigskip

\begin{example}
  Let $F=\left[
           \begin{array}{cccc}
             1 & 1 & z & z^2 \\
             1 & z & z^2 & z^3 \\
             1 & 0 & 0 & 0 \\
           \end{array}
         \right],
  $ then one can verify $\hbox{Rank}F =2$. Thus, if $F=\Theta G$ is the inner-outer factorization, then $\Theta \in H^\infty_{M_{3\times2}}$ and $G \in H^2_{M_{2\times 4}}$.
\end{example}

\bigskip

Let $A$ be an index set and let $\Theta_i\in H^\infty_{M_{n\times m_i}}$ be matrix inner functions for $i\in A$.  Since $\Theta_i H^2_{\mathbb C^{m_i}}$ is an invariant subspace of $H^2_{\mathbb C^n}$ for the shift operator $S$  for each $i\in A$, it is clear that $\bigvee_{i \in A} \Theta_i H^2_{\mathbb C^{m_i}}$ and $ \bigcap_{i\in A} \Theta_iH^2_{\mathbb C^{m_i}}$ are invariant subspaces for $S$ as well. Therefore, by Beurling-Lax-Halmos Theorem, there exist natural number $l, l^\prime$ $ (\le n)$ and matrix inner functions $\Theta\in H^\infty_{M_{n\times l}}$, $ \Theta^\prime \in H^\infty_{M_{n\times l^\prime}}$ such that
$$
\bigvee_{i \in A} \Theta_i H^2_{\mathbb C^{m_i}}=\Theta H^2_{\mathbb C^l} \hbox{ and } \bigcap_{i\in A} \Theta_iH^2_{\mathbb C^{m_i}} = \Theta^\prime H^2_{\mathbb C^{l^\prime}}.$$ These matrix inner functions $\Theta$  and $\Theta^\prime $ are defined to be the \textit{greatest common divisor} (GCD) and the \textit{least common multiple} (LCM) of the family of inner functions $\{\Theta_i | i\in A \}$, respectively.

\begin{theorem}\label{lcm}
For matrix inner functions $\Theta_i \in H^\infty_{M_{n\times m_i}}$ (for $i= 1,\cdots, r$), let $\Theta \in H^\infty_{M_{n\times l}}$ be $LCM \{\Theta_1,\cdots, \Theta_r\}$. Then
$$ n -\sum_{i=1}^r(n-m_1) \le l \le \min\{m_1, \cdots, m_r\}.
$$
\end{theorem}

\begin{proof}
By Corollary \ref{gen set cor}, for each $i=1,\cdots, r$, there exists a natural number $s_i\, (\le n)$ and an $s_i\times n$ matrix function $\Phi_i \in L^2_{M_{s_i \times n}}$ satisfying $\ker H_{\Phi_i}=\Theta_i H^2_{\mathbb C^{m_i}}$($i=1,\cdots, r$). Note, by Theorem \ref{ker of block Hankel}, that $\Phi_i$ has $n-m_i$ independent columns modulo Nevanlinna class.
Define a matrix function $\Phi:=\left[
                                              \begin{array}{c}
                                               \Phi_1 \\
                                                \vdots\\
                                               \Phi_r\\
                                              \end{array}
                                            \right]
$, then it is clear that
$$
\begin{array}{ll}
 \ker H_{\Phi}& =\ker H_{\Phi_1}\cap\cdots \cap \ker H_{\Phi_r} \\
   & = \Theta_1 H^2_{\mathbb C^{m_1}}\cap\cdots \cap \Theta_r H^r_{\mathbb C^{m_r}} \\
   & = \Theta H^2_{\mathbb C^{l}}
\end{array}.
$$
By the construction of $\Phi$, the maximum number of independent columns of $\Phi$ is at most $\Sigma_{i=1}^r(n-m_i)$ and at least $\max \{n-m_1,\cdots, n-m_r\}$. Therefore, using Theorem \ref{ker of block Hankel}, we conclude that $$n -\sum_{i=1}^r(n-m_1) \le l \le n-\max\{n-m_1,\cdots, n-m_r\}= \min\{m_1,\cdots, m_r\},$$ as desired.
The proof is complete.
\end{proof}

\bigskip
As an immediate consequence of the above theorem, we have
\begin{corollary}
  If $\Theta_1$ and $\Theta_2$ be matrix inner functions of size $(n\times n)$ and $(n\times m)$, respectively, then $\Theta:=LCM\{\Theta_1,\Theta_2\}$ is an $n\times m$ matrix (inner) function.
\end{corollary}

\bigskip

For GCD's of matrix inner functions, we have

\begin{theorem}\label{gcd}
 For matrix inner functions $\Theta_i \in H^\infty_{M_{n\times m_i}}$ for $i=1,\cdots r$, let $\Theta \in H^\infty_{M_{n\times l}}$ be $GCD\{\Theta_1,\cdots, \Theta_r\}$. Then
$$
\max\{ m_1,\cdots,m_r\} \le l \le \Sigma_{i=1}^{r}m_i.
$$
\end{theorem}

\begin{proof}
It is a simple observation that $\Theta_i H^2_{\mathbb C^{m_i}}$ is the shift invariant subspace generated by the columns of $\Theta_i$.
Therefore, $\bigvee_{i=1,\cdots, r}\Theta_i H^2_{\mathbb C^{m_i}}$ is the shift invariant subspace of $H^2_{\mathbb C^n}$ generated by the column vectors of $\Theta_1,\cdots,\Theta_r$.
Thus by Theorem \ref{shiftinvarsubsgeneratedby},
\begin{equation}\label{rank}
l=\hbox{Rank}[\Theta_1,\cdots, \Theta_r]= \hbox{ess} \sup_{\zeta \in \mathbb T}\hbox{rank}[\Theta_1(\zeta),\cdots, \Theta_r(\zeta)].
\end{equation}

Observe
\begin{equation}\label{rank1}
  \hbox{rank}\Theta_i(\zeta) \le \hbox{rank}[\Theta_1(\zeta),\cdots,  \Theta_r(\zeta)] \le \hbox{rank}\Theta_1(\zeta) +\cdots + \hbox{rank}\Theta_r(\zeta)
\end{equation} for each $\zeta \in \mathbb T$. By taking the essential supremum of (\ref{rank1}), we have
\begin{equation}\label{rank2}
  \hbox{Rank}\Theta_i \le \hbox{Rank}[\Theta_1,\cdots,  \Theta_r] \le \hbox{ess}\sup_{\zeta \in \mathbb T}(\hbox{rank}\Theta_1(\zeta) +\cdots + \hbox{rank}\Theta_r(\zeta)).
\end{equation}
Note the following is clear :
\begin{equation}\label{rank3}
  \hbox{ess}\sup_{\zeta \in \mathbb T}(\hbox{rank}\Theta_1(\zeta) +\cdots + \hbox{rank}\Theta_r(\zeta))\le
\hbox{Rank}\Theta_1 +\cdots + \hbox{Rank}\Theta_r.
\end{equation}
By (\ref{rank}), (\ref{rank2}) and (\ref{rank3}), we have
\begin{equation}\label{Rank ineq}
 m_i \le l \le\Sigma_{i=1}^r m_i.
\end{equation}
Since (\ref{Rank ineq}) holds for each $i=1,\cdots,r,$
we have $$\max\{m_1,\cdots,m_r\}\le l \le \Sigma_{i=1}^r m_i,$$
as desired.
\end{proof}

\begin{corollary}
  If $\Theta_1$ and $\Theta_2$ be matrix inner functions of size $(n\times n)$ and $(n\times m)$, respectively, then $\Theta:=GCD\{\Theta_1, \Theta_2\}$ is an $n\times n$ matrix (inner) function.
\end{corollary}

\begin{proof}
Let $\Theta = H^\infty_{M_{n\times l}}$, then by Theorem \ref{gcd}, $n= \max\{n, m\}\le l \le n+m$. But, since $\Theta$ is an inner function, $l \le n$. Thus, $l=n.$
\end{proof}

\bigskip
We conclude the paper with an example.
\begin{example}
  Let $\Phi=\left[
               \begin{array}{ccc}
                 a_1 & 0&0 \\
                 0 & \overline {\theta_1}&0 \\
                 0&0& \overline {\theta_2}\\
               \end{array}
             \right]
  $
  and $\Psi=\left[
               \begin{array}{ccc}
               a_2 &  a_2& 0\\
               0 & \overline {\theta_3} & 0\\
               0&0&\overline {\theta_4}\\
               \end{array}
             \right]
  $, where $a_1, a_2 \in L^2\setminus \mathcal N^2$ and $\theta_1, \cdots, \theta_4$ are inner functions. A simple calculation shows that $$\ker H_{\Phi}=\{(0, \theta_1 g_1, \theta_2 h_1)^t | g_1, h_1 \in H^2 \}\hbox{ and } \ker H_{\Psi}= \{(-\theta_3 g_2, \theta_3 g_2, \theta_4 h_2)^t | g_2, h_2\in H^2\}.$$

  Using matrix inner functions, $$\ker H_{\Phi}=\Theta_1 H^2_{\mathbb C^2} ,\hbox{ where,
} \Theta_1= \left[
\begin{array}{cc}
 0 & 0 \\
 \theta_1 & 0 \\
 0 & \theta_2 \\
 \end{array}
 \right]
$$ and $$\ker H_{\Psi}=\Theta_2 H^2_{\mathbb C^2},\hbox{ where,
}\Theta_2 = \left[
\begin{array}{cc}
-\frac{1}{\sqrt 2}\theta_3 & 0 \\
\frac{1}{\sqrt 2}\theta_3\ & 0 \\
0 & \theta_4 \\
  \end{array}
  \right]
.$$ Let $\varphi_i$ and $\psi_i$ be the column vectors of $\Phi$ and $\Psi$, respectively.
Then one can easily verify $\hbox{ind}_{\mathcal N}\{\varphi_1,\varphi_2,\varphi_3\}=\hbox{ind}_{\mathcal N}\{\psi_1,\psi_2,\psi_3\}=1.$
Therefore the sizes of the inner matrix functions $\Theta_1$ and $\Theta_2$ are in accordance with Theorem \ref{ker of block Hankel}.
 Let $\Theta:= LCM\{\Theta_1,\Theta_2\}$ and $\Theta^\prime := GCD\{\Theta_1,\Theta_2\}$. To find $\Theta$, we need to figure out the set $\ker H_\Phi \cap \ker H_\Psi(=\Theta_1 H^2_{\mathbb C^2}\cap \Theta_2 H^2_{\mathbb C^2})$. By analyzing the equation $(0,\theta_1g_1, \theta_2h_1)^t=(-\theta_3g_2, \theta_3 g_2, \theta_4 h_2)^t$, we find that $$\ker H_\Phi \cap \ker H_\Psi = \{(0,0, \theta h)^t | h\in H^2\}=\left[
                                                 \begin{array}{c}
                                                   0 \\
                                                   0 \\
                                                   \theta \\
                                                 \end{array}
                                               \right] H^2,$$
 where, $\theta:= LCM\{\theta_2,\theta_4\}$. Hence, $\Theta= \left[
                                                 \begin{array}{c}
                                                   0 \\
                                                   0 \\
                                                   \theta \\
                                                 \end{array}
                                               \right]
$. Note that if we set $\Omega:=\left[
                                  \begin{array}{c}
                                    \Phi \\
                                    \Psi \\
                                  \end{array}
                                \right]
$, then one can verify that the independency(mod $\mathcal N$) of the columns of $\Omega$ is $2$ and $\ker H_\Omega = \ker H_\Phi \cap \ker H_\Psi = \Theta H^2$, which also is in accordance with Theorem \ref{ker of block Hankel}.

To discuss the size of $\Theta^\prime=\hbox{GCD}\{\Theta_1,\Theta_2\}$, note that $\Theta_1 H^2_{\mathbb C ^2}\bigvee \Theta_2 H^2_{\mathbb C^2}$ is the shift-invariant subspace of $H^2_{\mathbb C^3}$ generated by the columns of $\Theta_1$ and $\Theta_2$. It is easy to verify $\hbox{Rank}[\Theta_1, \Theta_2]=3$ because the first three columns of $[\Theta_1, \Theta_2]$ has a nonzero determinant. Thus, by Theorem \ref{shiftinvarsubsgeneratedby}, we can conclude that $\Theta^\prime$ is a $3\times 3$ square matrix inner function.
\end{example}

Dong-O Kang

Department of Mathematics, Chungnam National University, Daejeon, 34134, Korea

E-mail: dokang@cnu.ac.kr

\bigskip

\thanks{This work was supported by Basic Science Research Program through the National Research Foundation of Korea(NRF) funded by the Ministry of Science, ICT and Future Planning(NRF-2015R1C1A1A01053837).}

Mathematics Subject Classification (2010). 47B35, 47A05

\end{document}